\documentclass[11pt]{amsart}    

\usepackage{graphicx}
\usepackage{amssymb,amsmath,amsthm}
\usepackage{hyperref}
\usepackage{mathtools}
\usepackage{bm}
\usepackage{color}
\usepackage{amscd}
\usepackage{stmaryrd}
\usepackage{mathrsfs}
\usepackage{tikz}

\newtheorem{theorem}{Theorem}[section]

\newtheorem{lemma}[theorem]{Lemma}

\newtheorem{corollary}[theorem]{Corollary}
\theoremstyle{definition}

\newtheorem{remark}[theorem]{Remark}

\numberwithin{equation}{section}
\numberwithin{table}{section}
\numberwithin{figure}{section}

\title{Nodal auxiliary a posteriori error estimates}

\author{Yuwen Li}
\address{School of Mathematical Sciences, Zhejiang University, Hangzhou, Zhejiang 310058, China.}
\email{liyuwen@zju.edu.cn}
\author{Ludmil T. Zikatanov}
\address{Department of Mathematics, The Pennsylvania State University, University Park, PA 16802, USA.}
\email{ludmil.math@gmail.com}

\begin{document}

\begin{abstract}
  We introduce and explain key relations between a posteriori error
  estimates and subspace correction methods viewed as preconditioners
  for problems in infinite dimensional Hilbert spaces. We set the
  stage using the Finite Element Exterior Calculus and Nodal Auxiliary
  Space Preconditioning. This framework provides a systematic way to derive explicit residual estimators and estimators based on local
  problems which are upper and lower bounds of the true error. We show the applications to discretizations of curl-curl, grad-div, Hodge Laplacian problems, and linear elasticity. For singularly perturbed $H({\rm curl})$ and $H({\rm div})$ problems, we
  also obtain novel parameter-independent error estimators. The only ingredients needed are:
  well-posedness of the problem and the existence of regular
  decomposition on continuous level.
\end{abstract}

\maketitle

\section{Introduction}
Adaptive mesh refinement is an indispensable component in the design
of competitive numerical models, as it can resolve singularities and sharp gradients in
the solution and often can provide optimal discretizations with
respect to computational complexity and accuracy. Adaptivity is a broad and sophisticated research area that involves many technical ingredients. In this paper we focus on a posteriori error estimates, the building block of adaptive mesh refinement and error control.  

Our objective is on deriving a posteriori error estimators by relating
the error control in adaptive methods to operator preconditioning
techniques (cf.~\cite{GriebelOswald1995,ToselliWidlund2005,Xu1992,XuZikatanov2002}) for iterative methods. This leads to simple and transparent analysis and to new a
posteriori error estimation techniques. We choose to work within the Finite Element Exterior
Calculus (FEEC)~\cite{ArnoldFalkWinther2006,ArnoldFalkWinther2010,Hiptmair1999MCOM}, because it is a general
framework which allows for clear and concise presentation of the
results.

For stationary problems, adaptive finite element methods (AFEMs) is a 
quite mature field, and we refer to the classical texts in
adaptivity~\cite{AinsworthOden2000,Bank1996,BabuskaStrouboulis2001,NochettoSiebertVeeser2009,Repin2008,Verfurth2013} for
basic techniques and references. A posteriori error analysis of
$H(\operatorname{curl})$ and $H(\operatorname{div})$, and in general
$H({\rm d})$ problems, where ${\rm d}$ is an exterior derivative, has been a topic
of intensive research in the last two decades. Related to our work is
\cite{BHHW2000}, where a posteriori error estimate for an eddy current
curl-curl problem is proposed. The main tool there is a Helmholtz
decomposition and the resulting estimator hinges on material
parameters and convexity of the physical domain.  An improved analysis
of this estimator is presented in~\cite{Schoberl2008}, which shows
parameter-independent reliability of that estimator on Lipschitz
domains with the help
of local regular decomposition and regularized interpolation commuting
with the gradient, curl and divergence. These works were further generalized 
on de Rham complexes \cite{DemlowHirani2014} in FEEC, where new estimators controlling the error of discrete Hodge Laplacian are derived based on global regular
decomposition and commuting interpolation. Later several  error estimators for decoupled errors of the discrete Hodge Laplacian could be found in \cite{ChenWu2017,Li2019SINUM}.

To motivate the study of the relations between a posteriori error
estimates and preconditioning, we make the following simple
observation: both methods aim at developing ``computable'' approximations to the action of solution operators on some residuals living in an infinite- or finite-dimensional space
(see~\cite{LiZikatanov2021CAMWA} for an example of such techniques in the symmetric
and positive definite case). It is therefore natural to find
connections between these numerical techniques and explore such interaction to construct estimators. We shall show that  in this way two-sided and robust a posteriori estimators can be derived for a wide
class of partial differential equations posed on de Rham complexes.

To confirm the reliability and efficiency of our a posteriori error estimators, we require two main ingredients: (1) solvability
of the continuous problem (continuous inf-sup condition); and (2) existence of an $H^1$ regular
decomposition (cf.~\cite{BS1987,DHL1999,Hiptmair2002,PasciakZhao2002}) of vector  fields (again, only on the continuous level). With these tools
we are able to derive classical and new error estimators for any well-posed problems on
de Rham complexes, including the positive definite curl-curl and grad-div problems, as well as indefinite Hodge Laplacian and linear elasticity. Moreover,
we construct a regular decomposition designed for singularly
perturbed $H({\rm d})$ norms. This decomposition suggests a new and robust two-sided error estimator for $H({\rm d})$ problems w.r.t.~material parameters. As far as we know, such robust estimators even for $H({\rm curl})$ and $H({\rm div})$ problems could not be found in the literature. 

A key idea borrowed from preconditioning theory is the methodology of Nodal Auxiliary
Space preconditioning \cite{HiptmairXu2007}, which reduces solution process of complex discrete $H({\rm curl})$ and $H({\rm div})$ problems to  solving several Poisson-type problems discretized by nodal elements; see also \cite{Li2024FoCM} for its generalization on surfaces. Meanwhile we note that a posteriori error estimation is equivalent to preconditioning the Riesz representation $\mathcal{B}_\mathcal{V}$ of the underlying dual Hilbert space $\mathcal{V}^\prime$. With the help of $H^1$ regular decomposition and fictitious space lemma \cite{Nepomnyaschikh1992}, $\mathcal{B}_\mathcal{V}$ is shown to be spectrally equivalent to an operator formed by Riesz representations of several $H^{-1}$ spaces. Therefore, estimating residuals in $\mathcal{V}^\prime$ reduces to a posteriori estimates of $H^{-1}$ residuals, which has been well established in the literature. The same approach works for $H({\rm d})$ singularly perturbed problems by reducing it to a posteriori error estimates of the classical  singularly perturbed $H^1$ reaction diffusion equation. Our framework directly generalizes to indefinite systems because the Riesz representation of underlying infinite-dimensional space in that case is simply the Cartesian product of representations of a few well-studied spaces.

Compared with $H(\operatorname{curl})$, a posteriori error analysis of $H(\operatorname{div})$ was initiated earlier because of its application in mixed methods of standard elliptic problems. Some early works in that direction could be found in~\cite{BV1996,Carstensen1997} where mesh dependent norms, quasi-uniform meshes, and Helmholtz
decomposition are used. Error estimators for controlling the $L^2$ flux error are presented in \cite{Alonso1996,Carstensen2005,HolstLiMihalikSzypowski2020,HuangXu2012}. It is noted that a fully general error estimator and a convergent adaptive algorithm are derived in \cite{CNS2007} for the positive definite $H({\rm div})$ problem. Finally we remark that residual estimators is not the whole story of a posteriori error estimation. Readers are referred to  \cite{AinsworthOden2000,Bank1996,BankXu2003b,BraessSchoberl2008,ErnVohralik2015,Repin2008} and references therein for hierarchical basis, equilibrated, superconvergent recovery, and functional error estimators. 

The rest of the paper is organized as follows. In
Section~\ref{secab} we give some standard notation and preliminaries
and link preconditioning with error estimation. Section~\ref{secderham}
introduces the basic FEEC notation and the minimum needed to set up finite element discretizations in the language of differential forms.  Section~\ref{secHgrad}
is devoted to the simplest case of error estimation for $H^{-1}$ residuals, the building block of our framework. Next, in Section~\ref{secHd} we give the
construction of estimators for the abstract $H({\rm d})$ problems, including singularly
perturbed, curl-curl, grad-div problems. We extend such results to product of Hilbert spaces in~Section \ref{secsaddlepoint} to present a
posteriori estimation for mixed  Hodge Laplacian and the linear elasticity with weakly
imposed symmetry.

\section{Abstract framework}\label{secab}

\subsection{Preliminaries}
For an arbitrary Hilbert space $\mathcal{V}$, let $(\cdot,\cdot)_\mathcal{V}$ denote its inner product, $\|\cdot\|_\mathcal{V}$ the $\mathcal{V}$-norm, $\mathcal{V}^\prime$ the dual space of $\mathcal{V},$ $\text{id}_\mathcal{V}$ the identity mapping on $\mathcal{V}.$ Let $\mathcal{B}_\mathcal{V}: \mathcal{V}^\prime\rightarrow\mathcal{V}$ be the Riesz representation of $\mathcal{V}^\prime$, namely,
$$(v,\mathcal{B}_{\mathcal{V}}r)_\mathcal{V}=\langle r,v\rangle,\quad\forall r\in \mathcal{V}^\prime,~\forall v\in \mathcal{V}.$$
Here $\langle\cdot,\cdot\rangle$ is the duality pairing between $\mathcal{V}'$ and $\mathcal{V}$. 
For a linear operator $g: \mathcal{V}_1\rightarrow\mathcal{V}_2$, its adjoint is denoted as $g^*: \mathcal{V}_2^\prime\rightarrow\mathcal{V}_1^\prime.$ We use $[\mathcal{V}]^i$ (resp.~$[\mathcal{V}]^{i\times j}$) to denote the Hilbert space of vectors (resp.~matrices with $i$ rows and $j$ columns) whose entries are contained in $\mathcal{V}.$

Given $f\in \mathcal{V}'$, consider the variational problem: $u\in\mathcal{V}$ such that \begin{equation}\label{var}
a(u,v)=\langle f,v\rangle,\quad\forall v\in\mathcal{V}.
\end{equation}
We assume that the bilinear form $a(\cdot,\cdot)$ is symmetric, bounded as follows
\begin{equation}\label{bdd}
\sup_{ 0\neq v \in \mathcal{V}}\sup_{0\neq w\in\mathcal{V}} \frac{a(v,w)}{\|v\|_\mathcal{V}\|w\|_\mathcal{V}}:=\alpha<\infty,
\end{equation}
and satisfies the following inf-sup condition
\begin{equation}\label{infsup}
\inf_{0\neq v\in \mathcal{V}}\sup_{0\neq w\in\mathcal{V}} \frac{a(v,w)}{\|v\|_\mathcal{V}\|w\|_\mathcal{V}}:=\beta>0.
\end{equation}
The bilinear form $a$ induces a bounded isomorphism $\mathcal{A}: \mathcal{V}\rightarrow \mathcal{V}'$ by
\begin{equation*}
\langle\mathcal{A}v, w  \rangle:= a(v,w),\quad\forall v,w\in \mathcal{V}.
\end{equation*}
Hence \eqref{var} is equivalent to the operator equation
\begin{equation}\label{operator}
    \mathcal{A}u=f.
\end{equation}
Obviously 
\eqref{bdd} and \eqref{infsup} are equivalent to that the operator norms satisfy
\begin{equation}\label{AAinvbd}
    \|\mathcal{A}\|_{\mathcal{V}\to\mathcal{V}^\prime}=\alpha, \quad \|\mathcal{A}^{-1}\|_{\mathcal{V}^{\prime}\to\mathcal{V}}=\beta^{-1}.
\end{equation}

Let us consider the approximation to~\eqref{var} by restricting it to a subspace $\mathcal{V}_h\subset\mathcal{V}$, namely: Find $u_h\in\mathcal{V}_h$ such that
\begin{equation}\label{disvar}
a(u_h,v)=\langle f,v\rangle,\quad\forall v\in\mathcal{V}_h.
\end{equation}
We assume that the discrete analogue of \eqref{infsup} is satisfied on $\mathcal{V}_h$, i.e.,
\begin{equation*}
\inf_{0\neq v\in \mathcal{V}_h}\sup_{0\neq w\in\mathcal{V}_h} \frac{a(v,w)}{\|v\|_\mathcal{V}\|w\|_\mathcal{V}}:= \beta_h>0.
\end{equation*}
It then follows from the Babu\v{s}ka--Brezzi theory that \eqref{disvar} admits a unique solution. 
In practice, $\mathcal{V}_h$ could be a finite element space based on a triangulation of a space domain.

\subsection{A posteriori error estimates by preconditioning}
We are interested in a posteriori error estimates of the form
\begin{align*}
    C_1\eta_h\leq\|u-u_h\|_\mathcal{V}\leq C_2\eta_h,
\end{align*}
where $\eta_h$ is  computable and $C_1, C_2$ are absolute constants. In finite element methods, $\eta_h$ can be used for adaptive mesh refinement. Let
\begin{align*}
    e&:=u-u_h\in\mathcal{V},\\
    \mathcal{R}&:=f-\mathcal{A}u_h\in\mathcal{V}^\prime.
\end{align*}
The variational formulation \eqref{disvar} implies 
\begin{equation*}
a(e,v)=\langle\mathcal{R},v\rangle,\quad\forall v\in\mathcal{V}.
\end{equation*}
or $\mathcal{A}e=\mathcal{R}$ in the operator form.

A posteriori error estimation then is to bound the norm of the true
error $e=\mathcal{A}^{-1}\mathcal{R}$ by computable error
indicators. Direct computation of $\|\mathcal{A}^{-1}\mathcal{R}\|_\mathcal{V}$
will be, in general, impossible or too expensive, since it requires to
compute the action of $\mathcal{A}^{-1}$ on $\mathcal{R}$.  Approximating such
action is known as \emph{preconditioning}; see, e.g.,
\cite{GriebelOswald1995,ToselliWidlund2005,Xu1992,XuZikatanov2002} for the
abstract frameworks on preconditioning. We now borrow some ideas
from this field and extend our results for elliptic problems \cite{LiZikatanov2021CAMWA} on $H^1$
to a posteriori error estimation of the possibly indefinite system
\eqref{disvar} with a more general space $\mathcal{V}$.

First, we need a bounded isomorphism (a \emph{preconditioner}) $\mathcal{B}: \mathcal{V}'\to\mathcal{V}$. Here $\mathcal{B}$ is assumed to be symmetric and positive definite (SPD), i.e., $\langle\cdot,\mathcal{B}\cdot\rangle$ is an inner product on $\mathcal{V}^\prime$, and the induced norm is denoted by $\|\cdot\|_{\mathcal{B}}$. Similarly, $\langle\mathcal{B}^{-1}\cdot,\cdot\rangle$ is an inner product on $\mathcal{V}$ and the induced norm is denoted as $\|\cdot\|_{\mathcal{B}^{-1}}.$
We say $\mathcal{B}$ is a preconditioner for $\mathcal{A}$ provided $\mathcal{B}\mathcal{A}: \mathcal{V}\rightarrow\mathcal{V}$ is a bounded isomorphism, i.e.,
\begin{equation*}
  \|\mathcal{B}\mathcal{A}\|_{\mathcal{B}^{-1}}<\infty,\quad\|(\mathcal{B}\mathcal{A})^{-1}\|_{\mathcal{B}^{-1}}<\infty.
\end{equation*}

\subsection{Estimating  the residual}
The next lemma shows that the existence of a preconditioner $\mathcal{B}$ naturally yields a two-sided estimate on $\|e\|_{\mathcal{B}^{-1}}$.
\begin{lemma}\label{reserror}
We have the following two sided bound
\begin{equation*}
\|\mathcal{B}\mathcal{A}\|_{\mathcal{B}^{-1}}^{-1}\|\mathcal{R}\|_\mathcal{B}\leq \|e\|_{\mathcal{B}^{-1}}\leq \|(\mathcal{B}\mathcal{A})^{-1}\|_{\mathcal{B}^{-1}}\|\mathcal{R}\|_\mathcal{B}.
\end{equation*}
\end{lemma}
\begin{proof}
Using the relation $e=\mathcal{A}^{-1}\mathcal{R}$, we have
\begin{align*}
    \|e\|_{\mathcal{B}^{-1}}=\|\mathcal{A}^{-1}\mathcal{B}^{-1}\mathcal{B}\mathcal{R}\|_{\mathcal{B}^{-1}}\leq\|(\mathcal{B}\mathcal{A})^{-1}\|_{\mathcal{B}^{-1}}\|\mathcal{B}\mathcal{R}\|_{\mathcal{B}^{-1}}.
\end{align*}
On the other hand, 
\begin{align*}
    \|\mathcal{R}\|_{\mathcal{B}}=\|\mathcal{B}\mathcal{A}e\|_{\mathcal{B}^{-1}}\leq\|\mathcal{B}\mathcal{A}\|_{\mathcal{B}^{-1}}\|e\|_{\mathcal{B}^{-1}}.
\end{align*}
The proof is complete.
\end{proof}

Hence we obtain an error estimator provided  the norm of the residual $\|\mathcal{R}\|_\mathcal{B}$ can be efficiently approximated.
It turns out that the \emph{fictitious space lemma} (see~\cite{Nepomnyaschikh1992}) is a natural way to
estimate the efficiency of a preconditioner and will be a powerful
tool in the current framework. 
\begin{lemma}[Fictitious Space Lemma]\label{FSP}
Let $\bar{\mathcal{V}}$ be a Hilbert space, $\bar{\mathcal{B}}: \bar{\mathcal{V}}^\prime\rightarrow \bar{\mathcal{V}}$ be a SPD operator.
Assume $\Pi: \bar{\mathcal{V}}\rightarrow\mathcal{V}$ is a surjective linear operator, and 
\begin{itemize}
\item For any $\bar{v}\in\bar{\mathcal{V}},$ it holds that $\|\Pi\bar{v}\|_{\mathcal{B}^{-1}}\leq c_0\|\bar{v}\|_{\bar{\mathcal{B}}^{-1}}$,
    \item 
For each $v\in\mathcal{V}$, there exists $\bar{v}\in \bar{\mathcal{V}}$ such that 
$$\Pi\bar{v}=v,\quad\|\bar{v}\|_{\bar{\mathcal{B}}^{-1}}\leq c_1\|v\|_{\mathcal{B}^{-1}}.$$
\end{itemize}
Then for $\widetilde{\mathcal{B}}:=\Pi\bar{\mathcal{B}}\Pi^*$ we have
\begin{equation*}
    c_0^{-2}\langle r,\widetilde{\mathcal{B}}r\rangle\leq \langle r,\mathcal{B}r\rangle\leq c_1^2\langle r,\widetilde{\mathcal{B}}r\rangle,\quad\forall r\in\mathcal{V}^\prime.
\end{equation*}
\end{lemma}
In our framework, the space $\bar{\mathcal{V}}$ will be formed by simple auxiliary spaces such that $\bar{\mathcal{B}}$ could be efficiently inverted. 

Clearly, the Riesz representation $\mathcal{B}_\mathcal{V}$ is a SPD operator satisfying $\langle\mathcal{B}_\mathcal{V}^{-1}\cdot,\cdot\rangle=(\cdot,\cdot)_{\mathcal{V}}$ and $\|\cdot\|_{\mathcal{B}_\mathcal{V}^{-1}}=\|\cdot\|_{\mathcal{V}}$.
Due to the boundedness \eqref{bdd} and the inf-sup condition
\eqref{infsup}, $\mathcal{B}_\mathcal{V}$ is a preconditioner for
$\mathcal{A}$, i.e.,
\begin{equation*}
    \|\mathcal{B}_\mathcal{V}\mathcal{A}\|_\mathcal{V}={\alpha},\quad\|(\mathcal{B}_\mathcal{V}\mathcal{A})^{-1}\|_\mathcal{V}=\beta^{-1},
\end{equation*}
respectively.
This fact is discovered in \cite{MardalWinther2011} and leads to a unified framework of preconditioning for linear systems arising from discretized PDEs. For our purpose, Lemma \ref{reserror} with $\mathcal{B}=\mathcal{B}_\mathcal{V}$ given above yields
\begin{equation}\label{reserror2}
    {\alpha}^{-1}\|\mathcal{R}\|_{\mathcal{B}_\mathcal{V}}\leq \|e\|_\mathcal{V}\leq \beta^{-1}\|\mathcal{R}\|_{\mathcal{B}_\mathcal{V}}.
\end{equation}

\section{Finite Element Exterior Calculus}\label{secderham}
In this section, we shall focus on the $H({\rm d})$ space of differential forms which includes the classical $H({\rm curl})$ and $H({\rm div})$ spaces as special examples.
For differential forms and exterior calculus,
we adopt the notation in \cite{ArnoldFalkWinther2006,ArnoldFalkWinther2010}.

\subsection{Continuous de Rham complex}
Let $\Omega\subset\mathbb{R}^n$ be a Lipschitz domain and $\Gamma\subseteq\partial\Omega$ be a closed set.
We assume the pair $(\Omega,\Gamma)$ is admissible in the sense of \cite{GolMM2011,Licht2019}. Given a subset $\Omega_0\subseteq\Omega,$ let $(\cdot,\cdot)_{\Omega_0}$ denote the $L^2$ inner product on $\Omega_0$ and $\|\cdot\|_{\Omega_0}$ the $L^2(\Omega_0)$ norm. For convenience, $(\cdot,\cdot)_{\Omega}$ and $\|\cdot\|_\Omega$ simplify as $(\cdot,\cdot)$ and $\|\cdot\|$, respectively. 

Let $dx^i$ be the $i$-th cotangent coordinate vector in $\mathbb{R}^n$ and $\wedge$ denote the wedge product of tensors.
We start with  $\Lambda^k(\Omega)$, the space of smooth differential $k$-forms on $\overline{\Omega},$ i.e.,
\begin{equation*}
    v\in\Lambda^k(\Omega)\Longleftrightarrow v=\sum_\alpha v_\alpha {\rm d}x^{\alpha_1}\wedge\cdots\wedge {\rm d}x^{\alpha_k},~v_\alpha\in C^\infty(\overline{\Omega}),
\end{equation*}
where the summation is taken over the set of  increasing multi-indices $\alpha=(\alpha_1,\ldots,\alpha_k)$ with $1\leq\alpha_1<\cdots<\alpha_k\leq n.$ In what {follows,} we say $\{v_\alpha\}$ are coefficients of $v\in \Lambda^k(\Omega).$ The $L^2$ norm and $H^1$ semi-norm on $\Omega$ are naturally defined coefficient-wise as
\begin{align*}
    \|v\|^2=\sum_{\alpha}\|v_\alpha\|^2,\quad|v|_{H^1}^2=\|\nabla v\|^2=\sum_{\alpha}\|\nabla v_\alpha\|^2.
\end{align*}
Other Sobolev norms could be defined in an analogous fashion.

Given an $(n-1)$-dimensional set $\gamma\subset\overline{\Omega},$ we use $\text{tr}=\text{tr}|_\gamma$ to denote the trace of differential forms on $\gamma$, which is the pullback of the inclusion $\gamma\hookrightarrow\overline{\Omega}$. For each index $k$, there exists exterior derivative ${\rm d}_k: \Lambda^k(\Omega)\rightarrow\Lambda^{k+1}(\Omega)$. Let $L^2\Lambda^k(\Omega)$ the space of $k$-forms with $L^2$ coefficients. The derivative ${\rm d}_k$ could be extended as a densely defined, unbounded operator ${\rm d}_k: L^2\Lambda^k(\Omega)\rightarrow L^2\Lambda^{k+1}(\Omega)$ and
we have ${\rm d}_{k+1}\circ {\rm d}_k=0$. 

The $H({\rm d})$ spaces of $k$-forms are
\begin{align*}
    H\Lambda^k(\Omega)&=\{v\in L^2\Lambda^k(\Omega): {\rm d}_kv\in L^2\Lambda^{k+1}(\Omega)\},\\
    \mathcal{V}^k&=\{v\in H\Lambda^k(\Omega): \text{tr} v=0\text{ on }\Gamma\}.
\end{align*}
It is noted that ${\rm d}_0$ could be identified with the usual gradient and  $\mathcal{V}^0$ is
$$\mathcal{H}_\Gamma:=\{v\in H^1(\Omega): v=0\text{ on }\Gamma\}.$$ 
In addition, $\mathcal{V}^1$ and $\mathcal{V}^2$ in $\mathbb{R}^3$ could be identified with $H({\rm curl})$ and $H({\rm div})$ spaces, respectively.
The de Rham complex is as follows
\begin{equation}\label{dR}
\begin{CD}
    \mathcal{V}^0@>{{\rm d}_0}>>\mathcal{V}^1@>{{\rm d}_1}>>\cdots@>{{\rm d}_{n-2}}>>\mathcal{V}^{n-1}@>{{\rm d}_{n-1}}>>\mathcal{V}^n.
    \end{CD}
\end{equation}
Let $\mathfrak{Z}^k=N({\rm d}_k)$ denote the kernel of ${\rm d}_k,$ $\mathfrak{B}^k=R({\rm d}_{k-1})$ the range of ${\rm d}_{k-1},$ $\perp$ the operation of taking $L^2$ orthogonal complement, and $\mathfrak{H}^k=\mathfrak{Z}^{k}\cap\mathfrak{B}^{k\perp}$ be the space of harmonic $k$-forms. 
Since ${\rm d}_k$ has a closed range (see, e.g.~\cite{ArnoldFalkWinther2010,GolMM2011}), we have the Hodge decomposition
\begin{equation}\label{Hodge}
    \mathcal{V}^k={\rm d}_{k-1}\mathcal{V}^{k-1}\oplus\mathfrak{H}^k\oplus\mathfrak{Z}^{k\perp},
\end{equation}
where $\oplus$ is the $L^2$ direct sum. For the same reason, the Poincar\'e inequality is true:
\begin{equation}\label{Poincare}
    \|v\|\leq c_P\|{\rm d}_kv\|,\quad v\in \mathfrak{Z}^{k\perp},
\end{equation}
where the constant $c_P>0$ depends on $\Omega, \Gamma.$

Let $H^1\Lambda^k(\Omega)$ denote the space of $k$-forms with $H^1$ coefficients. 
Assuming $\Gamma=\emptyset$ or $\Gamma=\partial\Omega$ and that $\Omega$ is convex or has smooth boundary, it is known that $\mathfrak{H}^k$ and $\mathfrak{Z}^{k\perp}$ are continuously embedded in $H^1\Lambda^k(\Omega)$ (see \cite{ABDG1998,ArnoldFalkWinther2006,Gaffney1951}): 
\begin{subequations}\label{qzH1}
    \begin{align}
&|q|_{H^1}\leq C_{\mathfrak{H}^k}\|q\|,\quad\forall q\in\mathfrak{H}^k,\\
&|z|_{H^1}\leq C_{\mathfrak{Z}^{k\perp}}\big(\|z\|+\|{\rm d}_kz\|\big),\quad\forall z\in\mathfrak{Z}^{k\perp},
    \end{align}
\end{subequations}
where the constants $C_{\mathfrak{H}^k}>0$, $C_{\mathfrak{Z}^{k\perp}}>0$ depend only on $\Omega, \Gamma.$ 

Let $\star: L^2\Lambda^k(\Omega)\rightarrow L^2\Lambda^{n-k}(\Omega)$ denote the Hodge star. The $L^2$ adjoint of ${\rm d}_{k-1}$ is the exterior coderivative $\delta_k: L^2\Lambda^k(\Omega)\rightarrow L^2\Lambda^{k-1}(\Omega)$ satisfying $\star\delta_k=(-1)^k{\rm d}_{n-k}\star$. For any subdomain $\Omega_0\subseteq\Omega$ and $v\in H^1\Lambda^k(\Omega_0)$, $w\in H^1\Lambda^{k+1}(\Omega_0),$ we shall frequently use the Stokes formula 
\begin{equation}\label{Stokes}
    ({\rm d}_kv,w)_{\Omega_0}=(v,\delta_{k+1} w)_{\Omega_0}+\int_{\partial\Omega_0}\text{tr} v\wedge\text{tr}\star w.
\end{equation}

\subsection{Discrete complex and proxy vector fields}
Let domain $\Omega$ be partitioned into a conforming and simplicial triangulation $\mathcal{T}_h$ aligned with $\Gamma.$ We use  $\mathcal{S}_h$ to denote the set of interior faces as well as boundary faces in $\partial\Omega\backslash\Gamma$. The mesh size functions $h$ and $h_s$ are defined as 
\begin{align*}
h|_T&=h_T:=\text{diam}(T),\quad\forall T\in\mathcal{T}_h,\\
h_s|_S&=h_S:=\text{diam}(S),\quad\forall S\in\mathcal{S}_h.    
\end{align*}
Let $\nu$ denote a piecewise constant vector field on $\mathcal{S}_h$ such that $\nu|_S=\nu_S$ is a unit normal to $S\in\mathcal{S}_h$. For a boundary face $S\subset\partial\Omega$, $\nu_S$ is chosen to be outward pointing. Let $(\cdot,\cdot)_{\mathcal{S}_h}$ and $\|\cdot\|_{\mathcal{S}_h}$ denote the $L^2$ inner product and norm on the skeleton $\mathcal{S}_h$, respectively.
The partition $\mathcal{T}_h$ is assumed to be shape regular in the sense that
\begin{equation*}
    \max_{T\in\mathcal{T}_h}\frac{\overline{\rho}_T}{\underline{\rho}_T}=C_{\text{shape}}<\infty,
\end{equation*} 
where $\overline{\rho}_T$ and $\underline{\rho}_T$ denote the radii of the circumscribed and inscribed spheres of an element $T,$ respectively.

We consider the Arnold--Falk--Winther element space  \cite{ArnoldFalkWinther2006,ArnoldFalkWinther2010}
\begin{equation*}
\begin{aligned}
    &\mathcal{P}_r\Lambda^k(\mathcal{T}_h)=\{v\in H\Lambda^k(\Omega): v|_T\in\mathcal{P}_r\Lambda^k(T),~\forall T\in\mathcal{T}_h\},\\
    &\mathcal{P}_{r+1}^-\Lambda^k(\mathcal{T}_h)=\mathcal{P}_r\Lambda^k(\mathcal{T}_h)+\iota \mathcal{P}_r\Lambda^k(\mathcal{T}_h),
\end{aligned}
\end{equation*}
where $\mathcal{P}_r\Lambda^k(T)$ is the space of $k$-forms on $T$ with polynomial coefficients of degree no greater than $r$, and $\iota$ is the interior product w.r.t.~the coordinate  vector $x=(x_1,\ldots,x_n).$ Let $\mathcal{V}^k_h\subset \mathcal{V}^k$ be 
\begin{equation}\label{Vhk}
\mathcal{V}^k_h=\left\{
\begin{aligned}
    &\mathcal{P}_r\Lambda^k(\mathcal{T}_h,\Gamma):=\mathcal{P}_r\Lambda^k(\mathcal{T}_h)\cap\mathcal{V}^k,\\
    &\mathcal{P}_{r+1}^-\Lambda^k(\mathcal{T}_h,\Gamma):=\mathcal{P}^-_{r+1}\Lambda^k(\mathcal{T}_h)\cap\mathcal{V}^k.
\end{aligned}\right.
\end{equation}
We have a discrete de Rham complex
\begin{equation*}
\begin{CD}
    \mathcal{V}_h^0@>{{\rm d}_0}>>\mathcal{V}_h^1@>{{\rm d}_1}>>\cdots@>{{\rm d}_{n-2}}>>\mathcal{V}_h^{n-1}@>{{\rm d}_{n-1}}>>\mathcal{V}_h^n
    \end{CD}
\end{equation*}
provided ${\rm d}_k\mathcal{V}_h^k\subseteq\mathcal{V}_h^{k+1}$ holds for each $k$.
The lowest order discrete complex is based on $\mathcal{V}^k_h=\mathcal{V}^{k,0}_h:=\mathcal{P}_1^-\Lambda^k(\mathcal{T}_h,\Gamma)$, the space of $k$-forms with incomplete piecewise linear coefficients.

In $\mathbb{R}^3,$ the $H({\rm curl})$ and $H({\rm div})$ spaces are
\begin{align*}
    \mathcal{V}^c&=\big\{v\in [L^2(\Omega)]^3: \nabla\times v\in[L^2(\Omega)]^3,~v\times \nu=0\text{ on }\Gamma\big\},\\
    \mathcal{V}^d&=\big\{v\in [L^2(\Omega)]^3: \nabla\cdot v\in L^2(\Omega),~v\cdot \nu=0\text{ on }\Gamma\big\}.
\end{align*}
Using proxy vector fields as in \cite{ArnoldFalkWinther2006}, we identify the de Rham complex \eqref{dR} ($n=3$) with a classical sequence 
\begin{equation}\label{identification}
\begin{CD}
    \mathcal{V}^0@>{{\rm d}_0}>>\mathcal{V}^1@>{{\rm d}_1}>>\mathcal{V}^2@>{{\rm d}_2}>>\mathcal{V}^3\\
    @VV\cong V  @VV\cong V  @VV\cong V @VV\cong V\\
    \mathcal{H}_\Gamma@>{\nabla}>>\mathcal{V}^c@>{\nabla\times}>>\mathcal{V}^d@>{\nabla\cdot}>>L^2(\Omega)
    \end{CD}
\end{equation}
Let $\mathcal{V}^c_h\subset \mathcal{V}^c$ be a N\'ed\'elec edge element space, and $\mathcal{V}^d_h\subset \mathcal{V}^d$ be  a Raviart--Thomas--N\'ed\'elec or Brezzi--Douglas--Marini finite element space; see, e.g.,  \cite{BrezziDouglasMarini1985,Monk2003,Nedelec1980,Nedelec1986,RaviartThomas1977}. We use $\mathcal{H}_h\subset\mathcal{H}_\Gamma$ (resp.~$L^2_h$) to denote the space of continuous (resp.~discontinuous) piecewise polynomials on $\mathcal{T}_h.$ The identification in the discrete level is
\begin{equation}\label{disidentification}
\begin{CD}
    \mathcal{V}_h^0@>{{\rm d}_0}>>\mathcal{V}_h^1@>{{\rm d}_1}>>\mathcal{V}_h^2@>{{\rm d}_2}>>\mathcal{V}_h^3\\
    @VV\cong V  @VV\cong V  @VV\cong V @VV\cong V\\
    \mathcal{H}_h@>{\nabla}>>\mathcal{V}_h^c@>{\nabla\times}>>\mathcal{V}_h^d@>{\nabla\cdot}>>L^2_h
    \end{CD}
\end{equation}
where we assume $\nabla\mathcal{H}_h\subseteq\mathcal{V}_h^c$, $\nabla\times\mathcal{V}^c_h\subseteq\mathcal{V}_h^d$, $\nabla\cdot\mathcal{V}_h^d\subseteq L^2_h$.

Let $H^1\Lambda^k(\mathcal{T}_h)$ denote the space of $k$-forms with piecewise $H^1$ coefficients w.r.t.~$\mathcal{T}_h.$ Given a face $S\in\mathcal{S}_h$ and $v\in H^1\Lambda^k(\mathcal{T}_h)$, we use $\llbracket{\text{tr} v}\rrbracket$ to denote the jump of the trace of $v$ across $S.$ 

\section{A posteriori error estimates in \emph{H}({\rm grad})}\label{secHgrad}
The building block of our theory is $\mathcal{H},$  a closed subspace of $H^1\Lambda^k(\Omega).$ In this section, we make the  natural choice 
\begin{equation*}
    \mathcal{H}=\mathcal{H}^k:=\{v\in H^1\Lambda^k(\Omega): v=0\text{ on }\Gamma\}.
\end{equation*}
We note that $\mathcal{H}^k$ with $0\leq k\leq n-1$ is equivalent to the Cartesian product of $\binom{n}{k}$ copies of the scalar-valued space
$\mathcal{H}_\Gamma.$ The $\mathcal{H}$ inner product is
\begin{equation*}
    (v,w)_{\mathcal{H}}=(v,w)+(\nabla v,\nabla w),\quad\forall v,w\in\mathcal{H}.
\end{equation*}
Let $\mathcal{A}_{\mathcal{H}}=\mathcal{B}_{\mathcal{H}}^{-1}: \mathcal{H}\rightarrow\mathcal{H}^\prime$ be the operator associated with $(\cdot,\cdot)_{\mathcal{H}}$. We shall make use of the subspace $\mathcal{H}^{k,1}_h\subset\mathcal{H}^k$ of $k$-forms with continuous and piecewise linear coefficients.

Throughout the rest of this paper, we adopt the notation $C_1\preccurlyeq C_2$ (resp.~$C_1$ $\lesssim C_2$) provided $C_1\leq C_3C_2$ with $C_3$ being a generic constant dependent only on $\Omega$, $\Gamma,$ $C_{\text{shape}}$ (resp.~$\Omega$, $\Gamma,$ $C_{\text{shape}}$ and physical parameters). We say $C_1\approx C_2$ provided $C_1\preccurlyeq C_2$, $C_2\preccurlyeq C_1$; and $C_1\eqsim C_2$ provided $C_1\lesssim C_2$ and $C_2\lesssim C_1$. This notation generalizes to comparison between SPD operators in an obvious way.

\subsection{Residual estimator}
The estimation of  $\|r\|_{\mathcal{B}_{\mathcal{H}}}=\langle r,{\mathcal{B}_{\mathcal{H}}} r\rangle^\frac{1}{2}$ is well understood and essential to a posteriori estimates of AFEMs for second order elliptic and Stokes equations; see, e.g., \cite{Verfurth2013}. We consider $r\in\mathcal{H}^\prime$ with the following $L^2$ representation 
\begin{equation}\label{rl2}
    \langle r,v\rangle=(R(r),v)+(J(r),\text{tr}v)_{\mathcal{S}_h},\quad\forall v\in\mathcal{H},
\end{equation}
where $R(r)\in L^2(\Omega)$, and $J(r)\in L^2(\mathcal{S}_h)$ is piecewise $L^2$ on the $(n-1)$-dimensional skeleton $\mathcal{S}_h$. The Riesz representor of $r$ is $e(r)\in\mathcal{H}$ such that \begin{equation}\label{er}
    (e(r),v)_\mathcal{H}=\langle r,v\rangle,\quad\forall v\in \mathcal{H}.
\end{equation}

To estimate the dual norm $\|r\|_{\mathcal{B}_{\mathcal{H}}}$ of the residual, it is necessary to construct a regularized interpolation onto the lowest order finite element space with local approximation property. We present such an interpolation in the next lemma. The details of the construction and the proof are given in the appendix. For $T\in\mathcal{T}_h$ (resp.~$S\in\mathcal{S}_h$), let $\Omega_T^k$ (resp.~$\Omega^k_S$) denote the union of elements (including $T$) sharing at least one $k$-dimensional simplex with $T$ (resp.~$S$) in $\mathcal{T}_h.$ 
\begin{lemma}\label{Pik}
For $0\leq k\leq n-1$, there exists $\Pi^k_h: \mathcal{H}^k\rightarrow\mathcal{P}_1^-\Lambda^k(\mathcal{T}_h,\Gamma)$ such that for all $v\in \mathcal{H}^k$ and $T\in\mathcal{T}_h$, $S\in\mathcal{S}_h$, 
\begin{subequations}\label{approxPik}
\begin{align}
\|\Pi^k_hv\|_T&\preccurlyeq\|v\|_{\Omega_T^k},\label{l2bd}\\
    \|v-\Pi^k_hv\|_T&\preccurlyeq h_T|v|_{H^1(\Omega^k_T)},\label{approx1}\\
    \|{\rm tr}(v-\Pi^k_hv)\|_S&\preccurlyeq h_S^\frac{1}{2}|v|_{H^1(\Omega^k_S)}.\label{approx2}
\end{align}
\end{subequations}
\end{lemma}
If the lowest order space $\mathcal{V}_h^{k,0}=\mathcal{P}_1^-\Lambda^k(\mathcal{T}_h,\Gamma)$ is not under
consideration, we note that $\mathcal{V}_h^k$ contains the space of continuous and piecewise
linear polynomials $\mathcal{H}_h^{k,1}$. Then it suffices to use the
coefficient-wise Cl\'ement interpolation (see~\cite{Clement})
$\Pi^{\rm C}_h: L^2\Lambda^k(\Omega)\rightarrow\mathcal{H}_h^{k,1}$ satisfying
\begin{equation}\label{approxPic}
\|h^{-1}(v-\Pi^{\rm C}_hv)\|+\|h_s^{-\frac{1}{2}}(v-\Pi^{\rm C}_hv)\|_{\mathcal{S}_h}^2\preccurlyeq|v|_{H^1(\Omega)}^2
\end{equation}
for all $v\in \mathcal{H}^k$. In addition, $\Pi^{\rm C}_h$ coincides with $\Pi_h^k$ when $k=0.$

Let $Q_h$ (resp.~$Q_h^s$) denote the $L^2$ projection onto the space
of piecewise polynomials, not necessarily continuous, and of fixed degree on
$\mathcal{T}_h$ (resp.~$\mathcal{S}_h$).  The data oscillation
is
\begin{equation}\label{def:osc}
\text{osc}_h(r):=\|h(R(r)-Q_hR(r))\|+\|h_s^\frac{1}{2}(J(r)-Q^s_hJ(r))\|_{\mathcal{S}_h}.
\end{equation}
The next lemma is proved using standard arguments in a posteriori
error estimation in the space $H^1(\Omega)$. For clarity we sketch its proof. The assumption $r\in(\mathcal{V}_h^{k,0})^\prime$ is
necessary because the residual $r$ always acts on the lowest order space
$\mathcal{V}_h^{k,0}\subset\mathcal{V}^k$, which is not contained in 
$\mathcal{H}_h^{k,1}$ if $k\geq1$.
\begin{lemma}\label{resH1}
For $0\leq k\leq n-1,$ let $r\in\mathcal{H}^\prime\cap(\mathcal{V}_h^{k,0})^\prime$ be in the form \eqref{rl2} and assume $\langle r,v_h\rangle=0$ for $v_h\in\mathcal{V}_h^{k,0}$.
Then we have
\begin{align*}
    &\|hR(r)\|+\|h_s^\frac{1}{2}J(r)\|_{\mathcal{S}_h}-{\rm osc}_h(r)\\
    &\quad\preccurlyeq\|r\|_{\mathcal{B}_\mathcal{H}}\preccurlyeq\|hR(r)\|+\|h_s^\frac{1}{2}J(r)\|_{\mathcal{S}_h}.
\end{align*}
\end{lemma}
\begin{proof}
It follows from the Cauchy--Schwarz inequality and Lemma \ref{Pik} that
\begin{align*}
&\|r\|_{\mathcal{B}_\mathcal{H}}={\sup_{v\in\mathcal{H}, \|v\|_{\mathcal{H}}=1}\langle r,v\rangle=\sup_{v\in\mathcal{H}, \|v\|_{\mathcal{H}}=1}\langle r,v-\Pi_h^kv\rangle}\\
    &\leq\left(\|hR(r)\|^2+\|h_s^\frac{1}{2}J(r)\|^2_{\mathcal{S}_h}\right)^\frac{1}{2}\\
    &\times\left(\|h^{-1}(v-\Pi_h^kv)\|^2+\|h_s^{-\frac{1}{2}}\text{tr}(v-\Pi_h^kv)\|_{\mathcal{S}_h}^2\right)^\frac{1}{2}.
\end{align*}
Then using \eqref{approxPik}, we prove 
the upper bound.
The lower bound follows from \eqref{er} with the bubble function technique explained in~\cite{Verfurth1996}. 
\end{proof}

\subsection{Local Dirichlet problems}

Let $\{x_i\}_{i=1}^N$ denote the set of vertices in $\mathcal{T}_h$. Let $\phi_i$ be the continuous piecewise linear function that takes the value 1 at $x_i$ and 0 at other vertices for each $x_i$. For $1\leq i\leq N,$ we define ${\Omega}_i:= \operatorname{supp}\phi_i$ and have
\begin{equation}\label{ephi}
\sum_{i=1}^N\phi_i(x)=1,\quad\|\nabla \phi_i\|_{L^\infty(\Omega)}\eqsim(\text{diam}\Omega_i)^{-1}.
\end{equation}
Let $$\mathcal{H}^k_i:=\{v\in\mathcal{H}: v=0\text{ on }\Omega\backslash\overline{\Omega}_i\}$$ and $\mathcal{I}_h: \mathcal{H}^{k,1}_h\hookrightarrow \mathcal{H}$, $\mathcal{I}_i: \mathcal{H}^k_i\hookrightarrow \mathcal{H}$ denote inclusions. 
Let $\mathcal{A}_h=\mathcal{I}^*_h\mathcal{A}_\mathcal{H}\mathcal{I}_h: \mathcal{H}^{k,1}_h\rightarrow(\mathcal{H}_h^{k,1})^\prime$ and $\mathcal{A}_i=\mathcal{I}^*_i\mathcal{A}_\mathcal{H}\mathcal{I}_i:$ $\mathcal{H}^k_i\rightarrow(\mathcal{H}^k_i)^\prime$ with $1\leq i\leq N$. We take the product space $\bar{\mathcal{H}}=\mathcal{H}^{k,1}_h\times\mathcal{H}_1\times\cdots\times\mathcal{H}_N$ and define $\Pi=(\mathcal{I}_h,\mathcal{I}_1,\ldots,\mathcal{I}_N): \bar{\mathcal{H}}\rightarrow\mathcal{H}$ as $$\Pi(v_h,v_1,\ldots,v_N)=v_h+v_1+\ldots+v_N,\quad v_h\in\mathcal{H}^{k,1}_h,~v_i\in\mathcal{H}^k_i.$$
Using the shape regularity of $\mathcal{T}_h$ and Cauchy--Schwarz inequality, we have $$\|\Pi\bar{v}\|_{\mathcal{H}}\preccurlyeq\|\bar{v}\|_{\bar{\mathcal{H}}},\quad\forall\bar{v}\in\bar{\mathcal{H}}.$$
For each $v\in\mathcal{H}$, let $v_h=\Pi^{\rm C}_hv$,  $v_i=(v-v_h)\phi_i$, and $\bar{v}=(v_h,v_1,\ldots,v_N)$. It follows from \eqref{approxPic} and \eqref{ephi} that
$$\Pi\bar{v}=v,\quad\|\bar{v}\|_{\bar{\mathcal{H}}}\preccurlyeq\|{v}\|_{{\mathcal{H}}}.$$
Therefore taking $\bar{\mathcal{B}}=\mathcal{B}_{\bar{\mathcal{H}}}=\text{diag}(\mathcal{A}_h^{-1},\mathcal{A}_1^{-1},\ldots,\mathcal{A}_N^{-1}): \bar{\mathcal{H}}^\prime\rightarrow\bar{\mathcal{H}}$ and $\mathcal{B}=\mathcal{B}_{\mathcal{H}}$ in Lemma \ref{FSP} respectively,  the two-level additive Schwarz preconditioner $$\mathcal{B}_\mathcal{H}^a:=\mathcal{I}_h\mathcal{A}^{-1}_h \mathcal{I}^*_h+\sum_{i=1}^N\mathcal{I}_i\mathcal{A}_i^{-1}\mathcal{I}^*_i.$$ 
is shown to be spectrally equivalent to $\mathcal{B}_\mathcal{H}$, i.e.,
\begin{equation}\label{BHa}
    \langle r,\mathcal{B}_\mathcal{H}r\rangle\simeq\langle r,\mathcal{B}_\mathcal{H}^a r\rangle=\langle \mathcal{I}_h^*r,\mathcal{A}_h^{-1}\mathcal{I}_h^*r\rangle+\sum_{i=1}^N\langle \mathcal{I}_i^*r,\mathcal{A}_i^{-1}\mathcal{I}_i^*r\rangle,\quad\forall r\in \mathcal{H}^{\prime}.
\end{equation}
Hence we obtain the implicit error indicator based on local problems.
\begin{lemma}\label{esteta}
Let $r\in\mathcal{H}^\prime$ and $\langle r,v_h\rangle=0$ $\forall v_h\in\mathcal{H}^{k,1}_{h}$. We have
\begin{equation*}
    \langle r,\mathcal{B}_\mathcal{H}r\rangle\eqsim\sum_{i=1}^N\|\eta_i\|^2_{H^1(\Omega_i)},
\end{equation*}
where $\eta_i\in\mathcal{H}^k_i$ solves
\begin{equation*}
    (\eta_i,\phi)+(\nabla \eta_i,\nabla\phi)=\langle r,\phi\rangle,\quad\forall\phi\in\mathcal{H}^k_i.
\end{equation*}
\end{lemma}
\begin{proof}
Let $\eta_i=\mathcal{A}_i^{-1}\mathcal{I}_i^*r$ in \eqref{BHa}, which solves the local problem in the lemma. Then noticing $\mathcal{I}_h^*r=0$ and using \eqref{BHa} and $\langle \mathcal{I}_i^*r,\mathcal{A}_i^{-1}\mathcal{I}_i^*r\rangle=\|\eta_i\|^2_{H^1(\Omega_i)}$, we complete the proof.
\end{proof}

The error estimator in Lemma \ref{esteta} was first proposed in \cite{BabuskaRheinboldt1978}
and explained using operator preconditioning in \cite{LiZikatanov2021CAMWA}. Here we
derive it from a new point of view, utilizing Lemma~\ref{FSP} (the fictitious space lemma). To make the local problem in Lemma \ref{esteta} solvable in practice, one could replace $\mathcal{H}^k_i$ with a (piecewise) polynomial subspace $\widetilde{\mathcal{H}}^k_i$. It can be shown that the estimator based on the discrete local subspace is still a two-sided bound of the residual provided $\widetilde{\mathcal{H}}^k_i$ contains suitable bubble functions; see, e.g., \cite{LiZikatanov2021CAMWA,Verfurth1996}.

\subsection{Robust error estimators in \emph{H}(grad)}
Let $\varepsilon$ and $\kappa$ be positive constants. The singularly perturbed $H({\rm grad})$ problem seeks $u\in\mathcal{H}_\Gamma$  such that
\begin{equation}\label{SPP}
    \varepsilon(\nabla u,\nabla v)+\kappa(u,v)=(f,v),\quad\forall v\in\mathcal{H}_\Gamma. 
\end{equation}
A posteriori error analysis of finite element methods for \eqref{SPP} is well established; see, e.g., \cite{Verfurth2013}. Here we consider the weighted $H^1$ space of $k$-forms $\mathcal{H}_w^k$, which consists of the same elements in $\mathcal{H}^k$ but equipped with the inner product and norm
\begin{align*}
    (v,\phi)_{\mathcal{H}_w^k}&:=\varepsilon(\nabla v,\nabla \phi)+\kappa(v,\phi),\\
    \|v\|_{\mathcal{H}_w^k}^2&:=\varepsilon\|\nabla v\|^2+\kappa\|v\|^2,\quad\forall v,\phi\in\mathcal{H}_w^k.
\end{align*}

The following lemma is essential for a posteriori error analysis of singularly perturbed $H({\rm d})$ problems. In this scenario, we shall assume $J(r)$ of certain residual functional $r$ is a piecewise polynomial on $\mathcal{S}_h$ because it is the case for
 singularly perturbed problems in Section \ref{sec:singularHd}.   The weighted mesh size functions and the technique in the following proof were used in \cite{Verfurth2013} for reaction diffusion equations in $H(\rm grad)$.
\begin{lemma}\label{resH1robust}
Let $\bar{h}=\min\{\varepsilon^{-\frac{1}{2}}h,\kappa^{-\frac{1}{2}}\}$ and $\bar{h}_s=\min\{\varepsilon^{-\frac{1}{2}}h_s,\kappa^{-\frac{1}{2}}\}$. Let $r\in(\mathcal{H}_w^k)^\prime\cap(\mathcal{V}_h^{k,0})^\prime$ be in the form \eqref{rl2}. Assume $J(r)$ is a piecewise polynomial on $\mathcal{S}_h$ and that $\langle r,v_h\rangle=0$ for any $v_h\in\mathcal{V}_h^{k,0}$.
Then we have
\begin{align*}
    &\|\bar{h} R(r)\|+\|\varepsilon^{-\frac{1}{4}}\bar{h}_s^\frac{1}{2}J(r)\|_{\mathcal{S}_h}-{\rm osc}_{\bar{h}}(r)\\
    &\quad\preccurlyeq\|r\|_{\mathcal{B}_{\mathcal{H}_w^k}}\preccurlyeq\|\bar{h} R(r)\|+\|\varepsilon^{-\frac{1}{4}}\bar{h}_s^\frac{1}{2}J(r)\|_{\mathcal{S}_h},
\end{align*}
where ${\rm osc}_{\bar{h}}(r)$ is 
 the weighted data oscillation:  
\begin{equation}\label{def:oscbar}
    {\rm osc}_{\bar{h}}(r):=\|\bar{h}(R(r)-Q_hR(r))\|.
\end{equation}
\end{lemma}
\begin{proof}
Using \eqref{l2bd} and \eqref{approx1}, we have for $v\in\mathcal{H}^k,$ 
\begin{align*}
      &\|v-\Pi_h^kv\|_T\preccurlyeq h_T\|\nabla v\|_{\Omega_T^k}\leq\varepsilon^{-\frac{1}{2}}h_T\|v\|_{\mathcal{H}_w^k(\Omega_T^k)},\\
      &\|v-\Pi_h^kv\|_T\preccurlyeq \|v\|_{\Omega_T^k}\leq\kappa^{-\frac{1}{2}}\|v\|_{\mathcal{H}_w^k(\Omega_T^k)}.
\end{align*}
Therefore it holds that
\begin{equation}\label{l2wT}
    \|v-\Pi_h^kv\|_T\preccurlyeq \bar{h}_T\|v\|_{\mathcal{H}_w^k(\Omega_T^k)}.
\end{equation}
Similarly, it follows from the trace inequality
\begin{equation}
    \|\text{tr}\phi\|^2_S\preccurlyeq h^{-1}_S\|\phi\|_T^2+\|\phi\|_T\|\nabla \phi\|_T,\quad\forall \phi\in H^1\Lambda^k(T)
\end{equation}
with $S\subset\partial T$ and \eqref{l2wT} that 
\begin{equation}\label{l2wS}
    \|\text{tr}(v-\Pi_h^kv)\|_S\preccurlyeq\varepsilon^{-\frac{1}{4}}\bar{h}^\frac{1}{2}_s|_S\|v\|_{\mathcal{H}_w^k(\Omega_S^k)}.
\end{equation}
The rest of the proof follows from~\eqref{l2wT}, \eqref{l2wS} and
the arguments used in the proof of Lemma~\ref{resH1}.
\end{proof}
In addition, we mention that robust equilibrated estimators for \eqref{SPP} have been established in, e.g., \cite{AinsworthBabuska1999,SmearsVohralik2020}.

\begin{remark}
  In addition to error indicators provided by Lemmata \ref{resH1},
  \ref{esteta}, and \ref{resH1robust}, any a posteriori error estimate
  that estimates $H^{-1}$-norm of a residual can be used in the
  current framework. For instance, equilibrated residual estimators
  \cite{AinsworthOden2000,BankWeiser1985,MNS2003} are able to yield very tight bounds
  on the residual.
\end{remark}

\section{A posteriori error estimates in \emph{H}({\rm d})}\label{secHd}

In many important applications, the space $\mathcal{V}$ in \eqref{var} is more complicated than $\mathcal{H}_\Gamma$ or $\mathcal{H}^k$, the standard $H(\text{grad})$ space.
Let $\mathcal{W}^-$, $\mathcal{W}$ be auxiliary Hilbert spaces that are connected with $\mathcal{V}$ via bounded linear operators $\mathcal{D}: \mathcal{W}^-\rightarrow\mathcal{V}$ and $\mathcal{I}: \mathcal{W}\rightarrow\mathcal{V}$. We shall choose
$\mathcal{W}$, $\mathcal{W}^-$ in appropriate ways such that the Riesz representation operators $\mathcal{B}_\mathcal{W}$, $\mathcal{B}_{\mathcal{W}^-}$ are amenable to a posteriori error analysis. The next corollary is a direct consequence of Lemma \ref{FSP} and very useful for preconditioning $\mathcal{B}_{\mathcal{V}}$.
\begin{corollary}\label{FSP2}
Assume for any $v\in\mathcal{V}$, there exist $w^-\in\mathcal{W}^-$ and $w\in\mathcal{W}$, such that
\begin{align*}
    &v=\mathcal{D}w^-+\mathcal{I}w,\\
    &\|w^-\|_{\mathcal{W}^-}^2+\|w\|_{\mathcal{W}}^2\leq C_{\emph{stb}}\|v\|^2_{\mathcal{V}},
\end{align*}
where $C_{\emph{stb}}$ is a constant.
Let $$\mathcal{B}_a=\mathcal{D}\mathcal{B}_{\mathcal{W}^-}\mathcal{D}^*+\mathcal{I}\mathcal{B}_{\mathcal{W}}\mathcal{I}^*: \mathcal{V}^\prime\rightarrow\mathcal{V}.$$  
Then we have
\begin{equation*}
    \big(\|\mathcal{D}\|^2+\|\mathcal{I}\|^2\big)\langle r,\mathcal{B}_ar\rangle\leq \langle r, \mathcal{B}r\rangle\leq C_{\emph{stb}}\langle r,\mathcal{B}_ar\rangle,\quad\forall r\in\mathcal{V}^\prime.
\end{equation*}
\end{corollary}
\begin{proof}
In Lemma \ref{FSP}, we take $\bar{\mathcal{V}}=\mathcal{W}^-\times\mathcal{W}$, $\Pi: \bar{\mathcal{V}}\rightarrow\mathcal{V}$ as $\Pi(w^-,w)=\mathcal{D}w^-+\mathcal{I}w$ and complete the proof.
\end{proof}
The stable splitting in Corollary \ref{FSP2} is motivated by the
regular and the Helmholtz decompositions; see, e.g.,
\cite{ArnoldFalkWinther2000,DemlowHirani2014,Hiptmair2002}. Typically, when applying the result
from Corollary \ref{FSP2}, $\mathcal{D}$ will be a differential operator
and $\mathcal{I}$ will be a natural inclusion.

\subsection{Standard \emph{H}(d) problems}
Let 
$$\varepsilon, \kappa\in L^\infty(\Omega)\cap H^1(\mathcal{T}_h),\quad f\in H^1\Lambda^k(\mathcal{T}_h).$$ 
For $1\leq k\leq n-1$, we assume the following problem is well-posed
\begin{align*}
    \delta_{k+1}(\varepsilon {\rm d}_ku)+\kappa u&=f\quad\text{in }\Omega,\\
    \text{tr}u&=0\quad\text{on }\Gamma.
\end{align*}
The variational problem is: Find $u\in\mathcal{V}^k$ such that
\begin{equation}\label{Hdvar}
    (\varepsilon{\rm d}_ku,{\rm d}_kv)+(\kappa u,v)=(f,v),\quad\forall v\in\mathcal{V}^k.
\end{equation}
The discrete problem seeks $u_h\in\mathcal{V}^k_h$ such that
\begin{equation}\label{disHdvar}
    (\varepsilon{\rm d}_ku_h,{\rm d}_k v)+(\kappa u_h,v)=(f,v),\quad\forall v\in\mathcal{V}_h^k.
\end{equation}
The residual $\mathcal{R}^k\in (\mathcal{V}^k)^\prime$ is
\begin{equation*}
    \langle \mathcal{R}^k,v\rangle=(f,v)-(\varepsilon{\rm d} _ku_h,{\rm d}_k v)-(\kappa u_h,v),\quad\forall v\in\mathcal{V}^k.
\end{equation*}
It then follows from \eqref{reserror2} that
\begin{equation}\label{Hderr0}
    \|u-u_h\|^2_{\mathcal{V}^k}\eqsim\langle \mathcal{R}^k,\mathcal{B}_{\mathcal{V}^k}\mathcal{R}^k\rangle
\end{equation}
The next theorem states the regular decomposition, see \cite{DHL1999,DemlowHirani2014,GopQiu2012,Hiptmair2002,PasciakZhao2002}.
\begin{theorem}[Regular Decomposition]\label{reg}
For $1\leq k\leq n-1$ and any $v\in \mathcal{V}^k,$ there exist $\varphi\in \mathcal{H}^{k-1}, z\in\mathcal{H}^k,$ such that 
\begin{align*}
    &v={\rm d}_{k-1}\varphi+z,\\
    &\|z\|_{H^1}\leq C_{\emph{reg}}\|{\rm d}_kv\|,\\
    &\|\varphi\|_{H^1}\leq C_{\emph{reg}}\|v\|_{\mathcal{V}^k},
\end{align*}
where $C_{\emph{reg}}$ is a constant dependent only on $\Omega$, $\Gamma.$
\end{theorem}

Theorem \ref{reg} suggests a preconditioner for $\mathcal{B}_{\mathcal{V}^k}$. Let $\mathcal{V}=\mathcal{V}^k,$ $\mathcal{W}^-=\mathcal{H}^{k-1}$, $\mathcal{W}=\mathcal{H}^k,$ $\mathcal{D}={\rm d}_{k-1}: \mathcal{H}^{k-1}\rightarrow \mathcal{V}^k$, and $\mathcal{I}=\mathcal{I}_{\mathcal{H}^k}: \mathcal{H}^k\hookrightarrow \mathcal{V}^k$ be the inclusion in Corollary \ref{FSP2}, where the stable splitting assumption is readily confirmed  by Theorem \ref{reg}. Therefore, we obtain the spectral equivalence
\begin{equation}\label{Bka}
    \mathcal{B}_{\mathcal{V}^k}\eqsim{\rm d}_{k-1}\mathcal{B}_{\mathcal{H}^{k-1}}{\rm d}_{k-1}^*+\mathcal{I}_{\mathcal{H}^{k}}\mathcal{B}_{\mathcal{H}^k}\mathcal{I}_{\mathcal{H}^k}^*.
\end{equation}
Collecting \eqref{Hderr0} and \eqref{Bka}, the nodal auxiliary a posteriori error estimate for \eqref{disHdvar} reads
\begin{equation}\label{Hderr1}
    \|u-u_h\|^2_{\mathcal{V}^k}\eqsim\langle {\rm d}_{k-1}^*\mathcal{R}^k,\mathcal{B}_{\mathcal{H}^{k-1}}{\rm d}_{k-1}^*\mathcal{R}^k\rangle+\langle\mathcal{I}_{\mathcal{H}^k}^*\mathcal{R}^k,\mathcal{B}_{\mathcal{H}^k}\mathcal{I}_{\mathcal{H}^k}^*\mathcal{R}^k\rangle.
\end{equation}
In what follows, we use results in Section \ref{secHgrad} to estimate the $H^{-1}$ residuals $d_{k-1}^*\mathcal{R}^k$ and $\mathcal{I}_{\mathcal{H}^k}^*\mathcal{R}^k$. For each $T\in \mathcal{T}_h$ and $S\in\mathcal{S}_h$, let
\begin{equation}\label{RJk}
    \begin{aligned}
    &R^k_1|_T=\delta_k(f-\kappa u_h)|_T,\\
    &J^k_1|_S=\llbracket\text{tr}\star(f-\kappa u_h)\rrbracket_S,\\
    &R^k_2|_T=(f-\delta_{k+1}(\varepsilon {\rm d}_ku_h)-\kappa u_h)|_T,\\
    &J^k_2|_S=-\llbracket\text{tr}\star\varepsilon {\rm d}_ku_h\rrbracket_S.
\end{aligned}
\end{equation}
Using ${\rm d}_k\circ {\rm d}_{k-1}=0$ and the Stokes formula \eqref{Stokes} element-wise, we obtain the $L^2$ representation of ${\rm d}_{k-1}^*\mathcal{R}^k\in(\mathcal{H}^{k-1})^\prime$ 
\begin{equation}\label{dstarR}
    \langle {\rm d}_{k-1}^*\mathcal{R}^k,v\rangle=\langle\mathcal{R}^k,{\rm d}_{k-1}v\rangle=(R^k_1,v)+(J^k_1,\text{tr}v)_{\mathcal{S}_h},\quad\forall v\in\mathcal{H}^{k-1}.
\end{equation}
Similarly, the $L^2$ representation of $\mathcal{I}_{\mathcal{H}^k}^*\mathcal{R}^k\in(\mathcal{H}^k)^\prime$ is
\begin{equation}\label{IHR}
    \langle \mathcal{I}_{\mathcal{H}^k}^*\mathcal{R}^k,v\rangle=\langle \mathcal{R}^k,v\rangle=(R_2^k,v)+(J_2^k,\text{tr}v)_{\mathcal{S}_h},\quad\forall v\in\mathcal{H}^k.
\end{equation}
Collecting previous results, we obtain the first main result.
\begin{theorem}\label{Hdestres}
For $1\leq k\leq n-1$, we have
\begin{align*}
    &\|u-u_h\|_{\mathcal{V}^k}\lesssim\|hR^k_1\|+\|h_s^\frac{1}{2}J^k_1\|_{\mathcal{S}_h}+\|hR_2^k\|+\|h_s^\frac{1}{2}J_2^k\|_{\mathcal{S}_h}.
\end{align*}
and
\begin{align*}
    &\|hR^k_1\|+\|h_s^\frac{1}{2}J^k_1\|_{\mathcal{S}_h}+\|hR_2^k\|+\|h_s^\frac{1}{2}J_2^k\|_{\mathcal{S}_h}\\
    &\quad\lesssim\|u-u_h\|_{\mathcal{V}^k}+{\rm osc}_h({\rm d}_{k-1}^*\mathcal{R}^k)+{\rm osc}_h(\mathcal{I}_{\mathcal{H}^k}^*\mathcal{R}^k).
\end{align*}
\end{theorem}
\begin{proof}
Using ${\rm d}_{k-1}\mathcal{V}^{k-1,0}_h\subseteq\mathcal{V}_h^{k,0}\subseteq\mathcal{V}_h^{k}$ and \eqref{disHdvar}, we have $\langle {\rm d}_{k-1}^*\mathcal{R}^k,v_h\rangle=0$ $\forall v_h\in\mathcal{V}^{k-1,0}_h$. Meanwhile \eqref{disHdvar} implies $\mathcal{I}_{\mathcal{H}^k}^*\mathcal{R}^k$ vanishes in $\mathcal{V}^{k,0}_h\subseteq\mathcal{V}_h^k$. It then follows from Lemma \ref{resH1} that
\begin{align*}
    &\|hR^k_1\|+\|h_s^\frac{1}{2}J^k_1\|_{\mathcal{S}_h}-\text{osc}_h({\rm d}_{k-1}^*\mathcal{R}^k)\\
    &\quad\lesssim\|{\rm d}_{k-1}^*\mathcal{R}^k\|_{\mathcal{B}_{\mathcal{H}^{k-1}}}\lesssim\|hR^k_1\|+\|h_s^\frac{1}{2}J^k_1\|_{\mathcal{S}_h},\\
    &\|hR^k_2\|+\|h_s^\frac{1}{2}J^k_2\|_{\mathcal{S}_h}-\text{osc}_h(\mathcal{I}_{\mathcal{H}^k}^*\mathcal{R}^k)\\
    &\quad\lesssim\|\mathcal{I}_{\mathcal{H}^k}^*\mathcal{R}^k\|_{\mathcal{B}_{\mathcal{H}^k}}\lesssim\|hR^k_2\|+\|h_s^\frac{1}{2}J^k_2\|_{\mathcal{S}_h}.
\end{align*}
Combining previous inequalities with \eqref{Hderr1} completes the proof.
\end{proof}

Recall $\mathcal{H}^k_i:=\{v\in\mathcal{H}^k: v=0\text{ on }\Omega\backslash\overline{\Omega}_i\}$ and note that $\mathcal{H}_h^{k,1}\subseteq\mathcal{V}_h^k$ if $\mathcal{V}_h^k\neq\mathcal{P}_0^-\Lambda^k(\mathcal{T}_h,\Gamma).$ As a consequence of \eqref{Hderr1} and Lemma \ref{esteta}, we obtain a new implicit error indicator for the $H({\rm d})$ problem based on solving local  Dirichlet  problems.
\begin{theorem}\label{Hdesteta}
For $1\leq k\leq n-1,$ let $\mathcal{V}_h^k\neq\mathcal{P}_0^-\Lambda^k(\mathcal{T}_h,\Gamma).$  Then we have
\begin{equation*}
    \|u-u_h\|^2_{\mathcal{V}^k}\eqsim\sum_{i=1}^N\left(\|\eta_i\|^2_{H^1(\Omega)}+\|\zeta_i\|^2_{H^1(\Omega)}\right),
\end{equation*}
where $\eta_i\in\mathcal{H}^{k-1}_i$ solves
\begin{equation}\label{local1}
    (\eta_i,\phi)+(\nabla \eta_i,\nabla\phi)=(f-\kappa u_h,{\rm d}_{k-1}\phi),\quad\forall\phi\in\mathcal{H}^{k-1}_i,
\end{equation}
and 
$\zeta_i\in\mathcal{H}^k_i$ solves
\begin{equation}\label{local2}
    (\zeta_i,\phi)+(\nabla \zeta_i,\nabla\phi)=(f,\phi)-(\varepsilon {\rm d}_ku_h,{\rm d}_k\phi)-(\kappa u_h,\phi),\quad\forall\phi\in\mathcal{H}^k_i.
\end{equation}
\end{theorem}

\begin{remark}
{Although the explicit residual estimator in Theorem \ref{Hdestres} is simple, it is difficult to balance the element-wise residual $\|hR^k_i\|$ and inter-element residual $\|\sqrt{h_s}J^k_i\|_{\mathcal{S}_h}$, $i=1, 2$. Information and possible cancellation might be lost by lumping those terms using a single multiplicative constant; see Chapter 3 of \cite{AinsworthOden2000}. As for the implicit residual estimator in Theorem \ref{Hdesteta}, volume and face residuals are hidden and naturally balanced in the right side of \eqref{local1} and \eqref{local2}. However, the equivalence in Theorem \ref{Hdesteta} still relies on some unknown multiplicative constant. 
These observations seem to motivate, although beyond the scope of this work, a future investigation of the connections between equilibrated a posteriori error estimates with explicit constants and preconditioning by non-overlapping domain decomposition methods. }
\end{remark}

\subsection{Singularly perturbed \emph{H}(d) problems}\label{sec:singularHd}
In this subsection, we focus on the case when $\varepsilon$ and
$\kappa$ are positive constants.  If $\varepsilon\ll1$ or
$\kappa\gg1$, then \eqref{Hdvar} becomes singularly perturbed. In this case,
the estimators in Theorems \ref{Hdestres} and \ref{Hdesteta}
deteriorate as $\varepsilon\rightarrow0$ or $\kappa\rightarrow\infty$
because the continuity and inf-sup constants of \eqref{Hdvar} are not
uniformly bounded. We note
that \eqref{Hdvar} with $k=0$ reduces to the classical singularly
perturbed second order elliptic equation \eqref{SPP}. To derive robust a posteriori estimates, it is
necessary to use the weighted space $\mathcal{V}_w^k,$ which consists
of the same elements in $\mathcal{V}^k,$ but is equipped with the
weighted inner product
\begin{equation*}
    (v,\phi)_{\mathcal{V}_w^k}=\varepsilon ({\rm d}_kv,{\rm d}_k\phi)+\kappa(v,\phi),\quad\forall v,\phi\in\mathcal{V}_w^k.
\end{equation*}
As a result, we obtain that the continuity and stability constants of \eqref{Hdvar} equal $1$, i.e., we have
    \begin{equation}\label{Hderr0w}
    \|u-u_h\|^2_{\mathcal{V}_w^k}=\langle \mathcal{R}^k,\mathcal{B}_{\mathcal{V}_w^k}\mathcal{R}^k\rangle.
\end{equation}

Let $\kappa\mathcal{H}^{k-1}$ consist of the same elements in $\mathcal{H}^{k-1}$ and be equipped with the inner product $\kappa(\cdot,\cdot)_{\mathcal{H}^{k-1}}$. Let $\mathcal{I}_{\mathcal{H}_w^k}: \mathcal{H}_w^k\hookrightarrow \mathcal{V}_w^k$ be the inclusion. Using Theorem \ref{reg} and following the same analysis for \eqref{Bka}, it is straightforward to check that the uniform spectral equivalence holds if $\kappa\leq\varepsilon$: 
\begin{equation}\label{Bkaw1}
    \mathcal{B}_{\mathcal{V}^k_w}\approx {\rm d}_{k-1}\mathcal{B}_{\kappa\mathcal{H}^{k-1}}{\rm d}_{k-1}^*+\mathcal{I}_{\mathcal{H}_w^{k}}\mathcal{B}_{\mathcal{H}_w^k}\mathcal{I}_{\mathcal{H}_w^k}^*.
\end{equation}
However, the regular decomposition is not suitable for the singularly perturbed case ($\kappa\gg\varepsilon$). Alternatively, we present an new stable decomposition of $\mathcal{V}^k$. The proof is technical and is given in the appendix.
\begin{theorem}[Robust Regular Decomposition]\label{robustHodge}
For $1\leq k\leq n-1$ and any $v\in\mathcal{V}^k$, there exist $\varphi\in \mathcal{H}^{k-1}$ and $z\in\mathcal{H}^k,$ such that 
\begin{align*}
    &v={\rm d}_{k-1}\varphi+z,\\
    &\|\varphi\|_{H^1}+\|z\|\leq C_\emph{rb}\|v\|,\\
    &|z|_{H^1}\leq C_\emph{rb} \|v\|_{\mathcal{V}^k},
\end{align*}
where $C_\emph{rb}$ is a constant depending on $\Omega$, $\Gamma.$
\end{theorem}
In contrast to the classical regular decomposition, the $L^2$ norm of component $z$ in Theorem \ref{robustHodge} is controlled by $\|v\|$, which is crucial for studying the dominant reaction case, i.e. $\kappa\gg1$.
\begin{remark}
  The HX preconditioner in \cite{HiptmairXu2007} is shown to be spectrally equivalent to the discrete $\mathcal{H}_w^k$ norm and robust w.r.t.~$\kappa$ in $\mathbb{R}^3$. However, that analysis is based on the assumption that the space domain is 2-regular, due to required regularity of the components in the Helmholtz decomposition. In this regard, we point out that the new regular decomposition from Theorem~\ref{robustHodge} also helps to remove such regularity assumption in the HX preconditioner. 
  
\end{remark}

In Corollary \ref{FSP2}, we take $\mathcal{V}=\mathcal{V}_w^k,$ $\mathcal{W}^-=\kappa\mathcal{H}^{k-1}$, $\mathcal{W}=\mathcal{H}_w^k,$ $\mathcal{D}={\rm d}_{k-1}$, and $\mathcal{I}=\mathcal{I}_{\mathcal{H}_w^k}$ to be the inclusion. When $\kappa\geq\varepsilon$, Theorem \ref{robustHodge} implies that the assumption in Corollary \ref{FSP2} holds  with $\|\mathcal{D}\|$, $\|\mathcal{I}\|$, $C_{\text{stb}}$ independent of $\kappa, \varepsilon$. Therefore the same uniform spectral equivalence follows if $\kappa\geq\varepsilon$
\begin{equation}\label{Bkaw2}
    \mathcal{B}_{\mathcal{V}^k_w}\approx {\rm d}_{k-1}\mathcal{B}_{\kappa\mathcal{H}^{k-1}}{\rm d}_{k-1}^*+\mathcal{I}_{\mathcal{H}_w^{k}}\mathcal{B}_{\mathcal{H}_w^k}\mathcal{I}_{\mathcal{H}_w^k}^*.
\end{equation}
Using \eqref{Hderr0w}, \eqref{Bkaw1}, \eqref{Bkaw2} and $\mathcal{B}_{\kappa\mathcal{H}^{k-1}}=\kappa^{-1}\mathcal{B}_{\mathcal{H}^{k-1}}$, we obtain 
\begin{equation}\label{Hderrorw}
    \|u-u_h\|^2_{\mathcal{V}_w^k}\approx\kappa^{-1}\|{\rm d}_{k-1}^*\mathcal{R}^k\|^2_{\mathcal{B}_{\mathcal{H}^{k-1}}}+\|\mathcal{I}_{\mathcal{H}_w^k}^*\mathcal{R}^k\|_{\mathcal{B}_{\mathcal{H}_w^k}}^2.
\end{equation}
Note that the $L^2$ representation of $\mathcal{I}_{\mathcal{H}_w^k}^*\mathcal{R}^k$ is the same as $\mathcal{I}_{\mathcal{H}^k}^*\mathcal{R}^k$ in \eqref{RJk}.
We are now in a position to present the second main result.
\begin{theorem}\label{Hdresestkappa}
For $1\leq k\leq n-1,$ it holds that
\begin{align*}
    \|u-u_h\|_{\mathcal{V}_w^k}&\preccurlyeq\|\kappa^{-\frac{1}{2}} hR_1^k\|+{\|\kappa^{-\frac{1}{2}}h_s^\frac{1}{2} J_1^k\|_{\mathcal{S}_h}}+\|\bar{h} R_2^k\|+\|\varepsilon^{-\frac{1}{4}}\bar{h}_s^\frac{1}{2} J_2^k\|_{\mathcal{S}_h}.
\end{align*}
and
\begin{align*}
    &\|\kappa^{-\frac{1}{2}} hR_1^k\|+{\|\kappa^{-\frac{1}{2}}h^\frac{1}{2} J_1^k\|_{\mathcal{S}_h}}+\|\bar{h} R_2^k\|+\| \varepsilon^{-\frac{1}{4}}\bar{h}_s^\frac{1}{2} J_2^k\|_{\mathcal{S}_h}\\
    &\quad\preccurlyeq\|u-u_h\|_{\mathcal{V}_w^k}+\kappa^{-\frac{1}{2}}\emph{osc}_h({\rm d}_{k-1}^*\mathcal{R}^k)+\emph{osc}_{\bar{h}}(\mathcal{I}_{\mathcal{H}_w^k}^*\mathcal{R}^k).
\end{align*}
\end{theorem}
\begin{proof}
It follows from Lemma \ref{resH1robust} that 
\begin{align*}
    &\|\bar{h}R^k_2\|+\|\varepsilon^{-\frac{1}{4}}\bar{h}_s^\frac{1}{2}J^k_2\|_{\mathcal{S}_h}-\text{osc}_{\bar{h}}(\mathcal{I}_{\mathcal{H}_w^k}^*\mathcal{R}^k)\\
    &\quad\preccurlyeq\|\mathcal{I}_{\mathcal{H}^k_w}^*\mathcal{R}^k\|_{\mathcal{B}_{\mathcal{H}_w^k}}\preccurlyeq\|\bar{h}R^k_2\|+\|\varepsilon^{-\frac{1}{4}}\bar{h}_s^\frac{1}{2}J^k_2\|_{\mathcal{S}_h}.
\end{align*}
We complete the proof using \eqref{Hderrorw}, the above estimate, and the bound for $\|{\rm d}_{k-1}^*\mathcal{R}^k\|^2_{\mathcal{B}_{\mathcal{H}^{k-1}}}$ in the proof of Theorem \ref{Hdestres}.
\end{proof}
The $H({\rm d})$ singularly perturbed problem has not been investigated and the estimator in Theorem \ref{Hdresestkappa} is new. Even a posteriori estimates in the special case $k=1, 2$ in $\mathbb{R}^3$ could not be found in the literature.

\subsection{Examples}
Using the identifications \eqref{identification} and \eqref{disidentification}, problems \eqref{Hdvar} and \eqref{disHdvar} with $k=1$, $n=3$ translate into: Find $u\in\mathcal{V}^c$, $u_h\in\mathcal{V}_h^c$ such that
\begin{subequations}\label{max}
\begin{align}
    (\varepsilon\nabla\times u,\nabla\times v)+(\kappa u,v)&=(f,v),\quad\forall v\in\mathcal{V}^c,\label{maxvar}\\
    (\varepsilon\nabla\times u_h,\nabla\times v_h)+(\kappa u,v_h)&=(f,v_h),\quad\forall v_h\in\mathcal{V}_h^c.\label{dismaxvar}
\end{align}
\end{subequations}
Similarly, \eqref{Hdvar} and \eqref{disHdvar} with $k=2$, $n=3$ is to find $u\in\mathcal{V}^d$ and $u_h\in\mathcal{V}_h^d$ such that
\begin{subequations}\label{div}\begin{align}
    (\varepsilon\nabla\cdot u,\nabla\cdot v)+(\kappa u,v)&=(f,v),\quad\forall v\in\mathcal{V}^d,\label{divvar}\\
    (\varepsilon\nabla\cdot u_h,\nabla\cdot v_h)+(\kappa u_h,v_h)&=(f,v_h),\quad\forall v_h\in\mathcal{V}_h^d.\label{disdivvar}
\end{align}
\end{subequations}
For $T\in\mathcal{T}_h$ and $S\in\mathcal{S}_h$, the element residuals and face jumps in \eqref{RJk} are 
\begin{equation}\label{RJ1}
    \begin{aligned}
    &R^1_1|_T=-\nabla\cdot(f-\kappa u_h)|_T,\\
    &J^1_1|_S=\llbracket f-\kappa u_h\rrbracket_S\cdot\nu_S,\\
    &R^1_2|_T=(f-\nabla\times(\varepsilon\nabla\times u_h)-\kappa u_h)|_T,\\
    &J^1_2|_S=-\llbracket\varepsilon\nabla\times u_h\rrbracket_S\times\nu_S,
\end{aligned}
\end{equation}
when $k=1$; and for $k=2$ we have
\begin{equation}\label{RJ2}
    \begin{aligned}
    &R^2_1|_T=\nabla\times(f-\kappa u_h)|_T,\\
    &J^2_1|_S=\llbracket f-\kappa u_h\rrbracket_S\times\nu_S,\\
    &R^2_2|_T=(f+\nabla(\varepsilon\nabla\cdot u_h)-\kappa u_h)|_T,\\
    &J^2_2|_S=\llbracket\varepsilon\nabla\cdot u_h\rrbracket_S.
\end{aligned}
\end{equation}
Let $\text{osc}_h$ be a generic data oscillation. Then Theorem \ref{Hdestres} with $k=1$ and $k=2$ yields residual estimators for the Maxwell equation \eqref{max}
\begin{align*}
    &\|u-u_h\|_{H(\text{curl})}\lesssim\|hR^1_1\|+\|h_s^\frac{1}{2}J^1_1\|_{\mathcal{S}_h}\\
    &+\|hR_2^1\|+\|h_s^\frac{1}{2}J_2^1\|_{\mathcal{S}_h}\lesssim\|u-u_h\|_{H(\text{curl})}+\text{osc}_h,
\end{align*}
and the grad-div problem \eqref{div}
\begin{align*}
    &\|u-u_h\|_{H(\text{div})}\lesssim\|hR^2_1\|+\|h_s^\frac{1}{2}J^2_1\|_{\mathcal{S}_h}\\
    &+\|hR_2^2\|+\|h_s^\frac{1}{2}J_2^2\|_{\mathcal{S}_h}\lesssim\|u-u_h\|_{H(\text{div})}+\text{osc}_h,
\end{align*}
respectively. As mentioned in the introduction, the above two estimators are first presented in \cite{BHHW2000,CNS2007}. Let 
$$\mathcal{H}_i:=\{v\in\mathcal{H}_\Gamma: v=0\text{ on }\Omega\backslash\overline{\Omega}_i\}.$$
We note that $\mathcal{H}^k_i$ defined in Section \ref{secHgrad} is the Cartesian product of $\binom{n}{k}$ copies of $\mathcal{H}_i$. Theorem \ref{Hdesteta} with $k=1$ yields a new implicit error estimator for the Maxwell equation
\begin{equation*}
    \|u-u_h\|^2_{H(\text{curl})}\eqsim\sum_{i=1}^N\left(\|\eta_i\|^2_{H^1(\Omega)}+\|\zeta_i\|^2_{H^1(\Omega)}\right),
\end{equation*}
where $\eta_i\in\mathcal{H}_i$ and $\zeta_i\in[\mathcal{H}_i]^3$ solve
\begin{align*}
    &(\eta_i,\psi)+(\nabla \eta_i,\nabla\psi)=(f-\kappa u_h,\nabla\psi),\quad\forall\psi\in\mathcal{H}_i,\\
    &(\zeta_i,\phi)+(\nabla \zeta_i,\nabla\phi)=(f-\kappa u_h,\phi)-(\varepsilon\nabla\times u_h,\nabla\times\phi),\quad\forall\phi\in[\mathcal{H}_i]^3.
\end{align*}
Other implicit a posteriori error estimates for Maxwell equations could be found in \cite{IHV2008,RZ2005}.

Similarly, for the problem \eqref{div}, the error estimator is
\begin{equation*}
    \|u-u_h\|^2_{H(\text{div})}\eqsim\sum_{i=1}^N\left(\|\eta_i\|^2_{H^1(\Omega)}+\|\zeta_i\|^2_{H^1(\Omega)}\right),
\end{equation*}
where $\eta_i\in[\mathcal{H}_i]^3$ and $\zeta_i\in[\mathcal{H}_i]^3$ solve
\begin{align*}
    &(\eta_i,\psi)+(\nabla \eta_i,\nabla\psi)=(f-\kappa u_h,\nabla\times\psi),\quad\forall\psi\in[\mathcal{H}_i]^3,\\
    &(\zeta_i,\phi)+(\nabla \zeta_i,\nabla\phi)=(f-\kappa u_h,\phi)-(\varepsilon \nabla\cdot u_h,\nabla\cdot\phi),\quad\forall\phi\in[\mathcal{H}_i]^3.
\end{align*}

\subsubsection{Robust estimators for $H({\rm curl})$ and $H({\rm div})$}
To close this section, we discuss the application of the robust estimator in Theorem \ref{Hdresestkappa}. With the identifications \eqref{RJ1} and \eqref{RJ2} explained above, Theorem \ref{Hdresestkappa} with $k=1, n=3$ gives a novel robust residual estimator for the Maxwell equation \eqref{max}, that is,
\begin{equation}\label{robust1}
\begin{aligned}
    &{\|u-u_h\|_{\mathcal{V}_w^c}:=}\|\varepsilon^\frac{1}{2}\nabla\times(u-u_h)\|+\|\kappa^\frac{1}{2}(u-u_h)\|\\
    &\approx\|\kappa^{-\frac{1}{2}} hR_1^1\|+{\|\kappa^{-\frac{1}{2}}h_s^\frac{1}{2} J_1^1\|_{\mathcal{S}_h}}+\|\bar{h} R_2^1\|+\|\varepsilon^{-\frac{1}{4}}\bar{h}_s^\frac{1}{2} J_2^1\|_{\mathcal{S}_h}+{\rm osc}_1,
\end{aligned}
\end{equation}
{where ${\rm osc}_1$ is a data oscillation (see definitions in \eqref{def:osc} and \eqref{def:oscbar}):
\begin{align*}
{\rm osc}_1&=\kappa^{-\frac{1}{2}}\|h(\nabla\cdot f-Q_h\nabla\cdot f)\|+\|\bar{h}(f-Q_hf)\|\\
&+\kappa^{-\frac{1}{2}}\|h_s^\frac{1}{2}(\llbracket f\cdot\nu\rrbracket-Q_h^s\llbracket f\cdot\nu\rrbracket)\|_{\mathcal{S}_h}.
\end{align*}}
Similarly, Theorem \ref{Hdresestkappa} with $k=2, n=3$ yields another robust estimator for the grad-div \eqref{div} problem: 
\begin{equation}\label{robust2}
\begin{aligned}
    &\|\varepsilon^\frac{1}{2}\nabla\cdot(u-u_h)\|+\|\kappa^\frac{1}{2}(u-u_h)\|\\
    &\approx\|\kappa^{-\frac{1}{2}} hR_1^2\|+{\|\kappa^{-\frac{1}{2}}h_s^\frac{1}{2} J_1^2\|_{\mathcal{S}_h}}+\|\bar{h} R_2^2\|+\|\varepsilon^{-\frac{1}{4}}\bar{h}_s^\frac{1}{2} J_2^2\|_{\mathcal{S}_h}+{\rm osc}_2,
\end{aligned}
\end{equation}
{where ${\rm osc}_2$ is another data oscillation: 
\begin{align*}
{\rm osc}_2&=\kappa^{-\frac{1}{2}}\|h(\nabla\times f-Q_h\nabla\times f)\|+\|\bar{h}(f-Q_hf)\|\\
&+\kappa^{-\frac{1}{2}}\|h_s^\frac{1}{2}(\llbracket f\times\nu\rrbracket-Q_h^s\llbracket f\times\nu\rrbracket)\|_{\mathcal{S}_h}. 
\end{align*}}
The new error estimators in \eqref{robust1} and \eqref{robust2} are uniform  w.r.t.~$\varepsilon$ and $\kappa$. 

In \cite{Schoberl2008}, for the Maxwell equation \eqref{max}, Sch\"oberl used innovative commuting extension, smoothing, and quasi-interpolating operators to derive the residual estimator {
\begin{equation*}
\tilde{\eta}_h=\|\kappa^{-\frac{1}{2}}hR_1^1\|+\|\kappa^{-\frac{1}{2}}h_s^\frac{1}{2}J_1^1\|_{\mathcal{S}_h}+\| \varepsilon^{-\frac{1}{2}}hR_2^1\|+\|\varepsilon^{-\frac{1}{2}}h_s^\frac{1}{2}J_2^1\|_{\mathcal{S}_h},
\end{equation*}
which is a uniform upper bound of the true error $\|u-u_h\|_{\mathcal{V}_w^c}$. However, the multiplicative constant $\widetilde{C}>0$ in the lower bound $\widetilde{C}\tilde{\eta}_h\leq\|u-u_h\|_{\mathcal{V}_w^c}$ is indeed influenced by $\varepsilon$ and $\kappa.$ We confirm this observation by the following numerical experiment.}

\begin{table}[tbhp]
\centering
\begin{tabular}{|c|c|c|c|c|c|c|}
\hline
number of & $e_h $&$\eta_h$
 &$\tilde{\eta}_h$
&  $e_h$ &  $\eta_h$ &  $\tilde{\eta}_h$ \\
 elements&$\kappa=10^2$&$ \kappa=10^2$&$ \kappa=10^2$&$ \kappa=10^3$&$ \kappa=10^3$&$\kappa=10^3$\\
\hline
750 &1.32&10.9&26.1&3.93&30.0&394\\
6000      &6.75e-1&6.14&10.7&2.02&16.2&112\\
48000      &3.51e-1&3.72&5.26&1.04&9.20&35.7\\
384000      &1.86e-1&2.28&2.70&5.45e-1&5.62&15.1\\
\hline
ei & N/A &1.01e-1&6.22e-2 &N/A &1.16e-1 &2.33e-2\\
\hline
number of & $e_h $&$\eta_h$
 &$\tilde{\eta}_h$
&  $e_h$ &  $\eta_h$ &  $\tilde{\eta}_h$ \\
 elements&$\kappa=10^4$&$ \kappa=10^4$&$ \kappa=10^4$&$ \kappa=10^5$&$ \kappa=10^5$&$\kappa=10^5$\\
\hline
             
750 &12.3&93.4&1.15e+4&39.0&294&3.26e+5\\
6000  &6.34&50.2&3.04e+3&20.0&158&9.47e+4\\
48000  &3.24&27.2&7.94e+2&10.2&85.7&2.43e+4\\
384000  &1.66&14.7&2.14e+2&5.24&45.5&6.25e+3\\
\hline
ei & N/A &1.22e-1&3.80e-3 &N/A &1.23e-1 &3.94e-4\\
\hline
\end{tabular}

\caption{Convergence of finite element error and error estimators on nested grids, $\varepsilon=\kappa^{-1}$; `ei' of $\eta_h$ (resp.~$\tilde{\eta}_h$) is the mean value of $e_h/\eta_h$ (resp.~$e_h/\tilde{\eta}_h$) over all grid levels.}
\label{tab:Hcurl}
\end{table}

{
We test the performance of $\tilde{\eta}_h$ and the proposed robust error estimator
\begin{equation*}
\eta_h=\|\kappa^{-\frac{1}{2}} hR_1^1\|+{\|\kappa^{-\frac{1}{2}}h_s^\frac{1}{2} J_1^1\|_{\mathcal{S}_h}}+\|\bar{h} R_2^1\|+\|\varepsilon^{-\frac{1}{4}}\bar{h}_s^\frac{1}{2} J_2^1\|_{\mathcal{S}_h}
\end{equation*}
by the lowest order edge element discretization of \eqref{max} on $\Omega=(0,1)^3$ with $\Gamma=\partial\Omega$. The exact solution is $u(x_1,x_2,x_3)=\big(0,0,\sin(\pi x_1)\sin(\pi x_2)\big).$ A tetrahedral mesh sequence is generated by uniform refinement in $\mathbb{R}^3$. 
Convergence history of $e_h:=\|u-u_h\|_{\mathcal{V}_w^c}$, $\eta_h$, $\tilde{\eta}_h$ is recorded in Table \ref{tab:Hcurl}. Unlike the uniform effectiveness of $\eta_h$ for all $\varepsilon, \kappa$, the efficiency of $\tilde{\eta}_h$ deteriorates for small $\varepsilon$ and large $\kappa.$  In fact, the ratio $e_h/\tilde{\eta}_h$ tends to zero as $\varepsilon\ll\kappa$, i.e., $\tilde{\eta}_h$ seriously overestimates the true error.} 

\section{Applications to saddle point systems}\label{secsaddlepoint}
In this section, we apply the theory developed in Sections \ref{secab}, \ref{secHgrad}, \ref{secHd} to several important saddle point systems. In particular, the space $\mathcal{V}$ in \eqref{var} is chosen as a Cartesian product of several Hilbert spaces.

\subsection{Hodge Laplace equation}
Given an index $1\leq k\leq n$, we assume the harmonic space $\mathfrak{H}^k=\{0\}$ is trivial for simplicity.  
The Hodge Laplace equation with $f\in L^2\Lambda^k(\Omega)$ under the mixed boundary condition is
\begin{align*}
    \sigma-\delta_k u&=0,\quad\text{in }\Omega,\\
    {\rm d}_{k-1}\sigma+\delta_{k+1} {\rm d}_ku&=f,\quad\text{in }\Omega,\\
    \text{tr}u&=0,\quad\text{on }\Gamma,\\
    \text{tr}\star u=0,\quad\text{tr}\star  {\rm d}_ku&=0,\quad\text{on }\partial\Omega\backslash\Gamma.
\end{align*}
The derivative ${\rm d}_n$ vanishes as in the classical notation.
The variational Hodge Laplacian problem seeks $(\sigma,u)\in \mathcal{V}^{k-1}\times\mathcal{V}^k$ such that 
\begin{equation}\label{HL}
    \begin{aligned}
    (\sigma,\tau)-({\rm d}_{k-1}\tau,u)&=0,\quad\tau\in\mathcal{V}^{k-1},\\
    ({\rm d}_{k-1}\sigma,v)+({\rm d}_ku,{\rm d}_kv)&=(f,v),\quad v\in\mathcal{V}^k.
\end{aligned}
\end{equation}
The corresponding discrete problem is: Find $(\sigma_h,u_h)\in \mathcal{V}_h^{k-1}\times\mathcal{V}_h^k$ such that 
\begin{equation}\label{disHL}
    \begin{aligned}
    (\sigma_h,\tau)-({\rm d}_{k-1}\tau,u_h)&=0,\quad\tau\in\mathcal{V}_h^{k-1},\\
    ({\rm d}_{k-1}\sigma_h,v)+({\rm d}_ku_h,{\rm d}_kv)&=(f,v),\quad v\in\mathcal{V}_h^k.
\end{aligned}
\end{equation}
The well-posedness of \eqref{HL} and \eqref{disHL} is confirmed in \cite{ArnoldFalkWinther2010}.
It follows from \eqref{Bka} that the Riesz representation $\mathcal{B}_{\mathcal{V}^{k-1}\times\mathcal{V}^k}$ ($1\leq k\leq n-1$) decouples as
\begin{equation}\label{Bkk}
\begin{aligned}
    &\mathcal{B}_{\mathcal{V}^{k-1}\times\mathcal{V}^k}=\mathcal{B}_{\mathcal{V}^{k-1}}\times\mathcal{B}_{\mathcal{V}^k}\\
    &\approx({\rm d}_{k-2}\mathcal{B}_{\mathcal{H}^{k-2}}{\rm d}_{k-2}^*+\mathcal{I}_{\mathcal{H}^{k-1}}\mathcal{B}_{\mathcal{H}^{k-1}}\mathcal{I}_{\mathcal{H}^{k-1}}^*)\\
    &\times({\rm d}_{k-1}\mathcal{B}_{\mathcal{H}^{k-1}}{\rm d}_{k-1}^*+\mathcal{I}_{\mathcal{H}^{k}}\mathcal{B}_{\mathcal{H}^k}\mathcal{I}_{\mathcal{H}^k}^*).
    \end{aligned}
\end{equation}
When $k=n$, $\mathcal{V}^n$ is simply $L^2\Lambda^n(\Omega)$ and
\begin{equation}\label{Bkn}
    \mathcal{B}_{\mathcal{V}^{n-1}\times\mathcal{V}^n}\approx({\rm d}_{n-2}\mathcal{B}_{\mathcal{H}^{n-2}}{\rm d}_{n-2}^*+\mathcal{I}_{\mathcal{H}^{n-1}}\mathcal{B}_{\mathcal{H}^{n-1}}\mathcal{I}_{\mathcal{H}^{n-1}}^*)\times\text{id}_{\mathcal{V}^n}.
\end{equation}
The residual $\mathcal{R}=(\mathcal{R}_\sigma,\mathcal{R}_u)\in (\mathcal{V}^{k-1})^\prime\times (\mathcal{V}^k)^\prime$ is given as
\begin{equation}\label{R1R2}
\begin{aligned}
    &\langle \mathcal{R}_\sigma,\tau\rangle=-(\sigma_h,\tau)+({\rm d}_{k-1}\tau,u_h),\quad\forall\tau\in\mathcal{V}^{k-1},\\
    &\langle \mathcal{R}_u,v\rangle=(f,v)-({\rm d}_{k-1}\sigma_h,v)-({\rm d}_ku_h,{\rm d}_kv),\quad \forall v\in\mathcal{V}^k.
\end{aligned}
\end{equation}
Using \eqref{reserror2} and \eqref{Bkk}, it is shown that the true error when $1\leq k\leq n-1$  could be controlled by four $H^{-1}$ residuals 
\begin{equation}\label{HLerror}
\begin{aligned}
    &\|\sigma-\sigma_h\|^2_{\mathcal{V}^{k-1}}+\|u-u_h\|^2_{\mathcal{V}^k}\approx\langle\mathcal{R}_\sigma,\mathcal{B}_{\mathcal{V}^{k-1}}\mathcal{R}_\sigma\rangle+\langle\mathcal{R}_u,\mathcal{B}_{\mathcal{V}^{k}}\mathcal{R}_u\rangle\\
    &=\langle{\rm d}_{k-2}^*\mathcal{R}_\sigma,\mathcal{B}_{\mathcal{H}^{k-2}}{\rm d}_{k-2}^*\mathcal{R}_\sigma\rangle+\langle\mathcal{I}^*_{\mathcal{H}^{k-1}}\mathcal{R}_\sigma,\mathcal{B}_{\mathcal{H}^{k-1}}\mathcal{I}_{\mathcal{H}^{k-1}}^*\mathcal{R}_\sigma\rangle\\
    &+\langle {\rm d}_{k-1}^*\mathcal{R}_u,\mathcal{B}_{\mathcal{H}^{k-1}}{\rm d}_{k-1}^*\mathcal{R}_u\rangle+\langle\mathcal{I}^*_{\mathcal{H}^k}\mathcal{R}_u,\mathcal{B}_{\mathcal{H}^k}\mathcal{I}_{\mathcal{H}^k}^*\mathcal{R}_u\rangle.
    \end{aligned}
\end{equation}
Here for convenience, any quantity with index $k=-1$ vanishes. 
Similarly, in the case $k=n$, \eqref{Bkn} implies that
\begin{equation}\label{HLerrorn}
\begin{aligned}
    &\|\sigma-\sigma_h\|^2_{\mathcal{V}^{n-1}}+\|u-u_h\|^2_{\mathcal{V}^n}\\
    &\approx\langle {\rm d}_{n-2}^*\mathcal{R}_\sigma,\mathcal{B}_{\mathcal{H}^{n-2}}{\rm d}_{n-2}^*\mathcal{R}_\sigma\rangle+\langle\mathcal{I}^*_{\mathcal{H}^{n-1}}\mathcal{R}_\sigma,\mathcal{B}_{\mathcal{H}^{n-1}}\mathcal{I}_{\mathcal{H}^{n-1}}^*\mathcal{R}_\sigma\rangle+\|\mathcal{R}_u\|^2,
    \end{aligned}
\end{equation}
where $\mathcal{R}_u=f-{\rm d}_{n-1}\sigma_h\in L^2\Lambda^n(\Omega)$ in this case.
Using \eqref{R1R2} and the Stokes formula \eqref{Stokes} on each element, it follows that the $L^2$ representations of the four residuals in \eqref{HLerror} are
\begin{equation*}
\begin{aligned}
    \langle {\rm d}_{k-2}^*\mathcal{R}_\sigma,\xi\rangle&=(R^k_{H,1},\xi)+(J^k_{H,1},\text{tr}\xi)_{\mathcal{S}_h},\quad\forall \xi\in\mathcal{H}^{k-2},\\
    \langle \mathcal{I}^*_{\mathcal{H}^{k-1}}\mathcal{R}_\sigma,\tau\rangle&=(R^k_{H,2},\tau)+(J^k_{H,2},\text{tr}\tau)_{\mathcal{S}_h},\quad\forall \tau\in\mathcal{H}^{k-1},\\
    \langle {\rm d}_{k-1}^*\mathcal{R}_u,\tau\rangle&=(R^k_{H,3},\tau)+(J^k_{H,3},\text{tr}\tau)_{\mathcal{S}_h},\quad\forall \tau\in\mathcal{H}^{k-1},\\
    \langle \mathcal{I}^*_{\mathcal{H}^k}\mathcal{R}_u,v\rangle&=(R^k_{H,4},v)+(J^k_{H,4},\text{tr}v)_{\mathcal{S}_h},\quad\forall v\in\mathcal{H}^k,
    \end{aligned}
\end{equation*}
where the element residuals and face jumps are
\begin{align*}
    &R^k_{H,1}|_T=-\delta_{k-1}\sigma_h|_T,\quad J^k_{H,1}|_S=-\llbracket\text{tr}\star\sigma_h\rrbracket|_S,\\
    &R^k_{H,2}|_T=(-\sigma_h+\delta_k u_h)|_T,\quad J^k_{H,2}|_S=\llbracket\text{tr}\star u_h\rrbracket|_S,\\
    &R^k_{H,3}|_T=\delta_k(f-{\rm d}_{k-1}\sigma_h)|_T,\quad J^k_{H,3}|_S=\llbracket\text{tr}\star(f- {\rm d}_{k-1}\sigma_h)\rrbracket|_S,\\
    &R^k_{H,4}|_T=(f-{\rm d}_{k-1}\sigma_h-\delta_{k+1} {\rm d}_ku_h)|_T,\quad J^k_{H,4}|_S=-\llbracket\text{tr}\star {\rm d}_ku_h\rrbracket|_S,
\end{align*}
for all $T\in\mathcal{T}_h$, $S\in\mathcal{S}_h$.
Therefore estimating those residuals in \eqref{HLerror} and \eqref{HLerrorn} by Lemma \ref{resH1}, we obtain the residual estimator for the Hodge Laplacian. 
\begin{theorem}\label{Hodgeestres}
For $1\leq k\leq n-1$, we have
\begin{align*}
    \|\sigma-\sigma_h\|_{\mathcal{V}^{k-1}}+\|u-u_h\|_{\mathcal{V}^k}\preccurlyeq\sum_{i=1}^4\left(\|hR^k_{H,i}\|+\|h_s^\frac{1}{2}J^k_{H,i}\|_{\mathcal{S}_h}\right)
\end{align*}
and
\begin{align*}
    &\sum_{i=1}^4\left(\|hR^k_{H,i}\|+\|h_s^\frac{1}{2}J^k_{H,i}\|_{\mathcal{S}_h}\right)\preccurlyeq\|\sigma-\sigma_h\|_{\mathcal{V}^{k-1}}+\|u-u_h\|_{\mathcal{V}^k}\\
    &\quad+\emph{osc}_h({\rm d}_{k-2}^*\mathcal{R}_\sigma)+\emph{osc}_h(\mathcal{I}_{\mathcal{H}^{k-1}}^*\mathcal{R}_\sigma)+\emph{osc}_h({\rm d}_{k-1}^*\mathcal{R}_u)+\emph{osc}_h(\mathcal{I}_{\mathcal{H}^k}^*\mathcal{R}_u).
\end{align*}
For $k=n$, we have
\begin{align*}
    \|\sigma-\sigma_h\|_{\mathcal{V}^{n-1}}+\|u-u_h\|_{\mathcal{V}^n}\preccurlyeq\sum_{i=1}^2\left(\|hR^n_{H,i}\|+\|h_s^\frac{1}{2}J^n_{H,i}\|_{\mathcal{S}_h}\right)+\|f-{\rm d}_{n-1}\sigma_h\|
\end{align*}
and
\begin{align*}
    &\sum_{i=1}^2\left(\|hR^{n}_{H,i}\|+\|h_s^\frac{1}{2}J^{n}_{H,i}\|_{\mathcal{S}_h}\right)+\|f-{\rm d}_{n-1}\sigma_h\|\\
    &\preccurlyeq\|\sigma-\sigma_h\|_{\mathcal{V}^{n-1}}+\|u-u_h\|_{\mathcal{V}^n}+\emph{osc}_h({\rm d}_{n-2}^*\mathcal{R}_\sigma)+\emph{osc}_h(\mathcal{I}_{\mathcal{H}^{n-1}}^*\mathcal{R}_\sigma).
\end{align*}
\end{theorem}
\begin{proof}
{It suffices to verify that ${\rm d}_{k-2}^*\mathcal{R}_\sigma$, $\mathcal{I}^*_{\mathcal{H}^{k-1}}\mathcal{R}_\sigma$, ${\rm d}_{k-1}^*\mathcal{R}_u$, $\mathcal{I}^*_{\mathcal{H}^k}\mathcal{R}_u$ fulfill the assumption in Lemma \ref{resH1}. For example, by \eqref{R1R2}, \eqref{disHL} and ${\rm d}_{k-1}\circ{\rm d}_{k-2}=0$, for any $\omega_h\in\mathcal{V}_{h}^{k-2,0}$ we have
\begin{align*}
&\langle {\rm d}_{k-2}^*\mathcal{R}_\sigma,\omega_h\rangle=\langle \mathcal{R}_\sigma,{\rm d}_{k-2}\omega_h\rangle\\
&=-(\sigma_h,{\rm d}_{k-2}\omega_h)=-({\rm d}_{k-1}{\rm d}_{k-2}\omega_h,u_h)=0.
\end{align*}
The other three residuals could be checked in a similar way.}
\end{proof}
The residual estimator in Theorem \ref{Hodgeestres} was first derived in \cite{DemlowHirani2014} using commuting regularized interpolation, which could be avoided in our framework. In addition, a new implicit error estimator follows from \eqref{HLerror}, \eqref{HLerrorn}, Lemma \ref{esteta}, and $\mathcal{H}_h^{k,1}\subseteq\mathcal{V}_h^k$ if $\mathcal{V}_h^k\neq\mathcal{P}_0^-\Lambda^k(\mathcal{T}_h,\Gamma).$
\begin{theorem}\label{Hodgeesteta}
Let $\mathcal{V}_h^k\neq\mathcal{P}_0^-\Lambda^k(\mathcal{T}_h,\Gamma)$ with $1\leq k\leq n-1.$ Then for $1\leq k\leq n-1,$ we have
\begin{align*}
    &\|\sigma-\sigma_h\|^2_{\mathcal{V}^{k-1}}+\|u-u_h\|^2_{\mathcal{V}^k}\\
    &\approx\sum_{i=1}^N\left(\|\eta^\sigma_i\|^2_{H^1(\Omega_i)}+\|\zeta^\sigma_i\|^2_{H^1(\Omega_i)}+\|\eta^u_i\|^2_{H^1(\Omega_i)}+\|\zeta^u_i\|^2_{H^1(\Omega_i)}\right),
\end{align*}
where $\eta^\sigma_i\in\mathcal{H}^{k-2}_i$, $\zeta^\sigma_i, \eta^u_i\in\mathcal{H}^{k-1}_i$, $\zeta^u_i\in\mathcal{H}^k_i$ solve
\begin{align*}
    (\eta^\sigma_i,\xi)+(\nabla \eta^\sigma_i,\nabla\xi)&=-(\sigma_h,{\rm d}_{k-2}\xi),\\
    (\zeta^\sigma_i,\tau)+(\nabla \zeta^\sigma_i,\nabla\tau)&=-(\sigma_h,\tau)+({\rm d}_{k-1}\tau,u_h),\\
    (\eta^u_i,v)+(\nabla \eta^u_i,\nabla \tau)&=(f,{\rm d}_{k-1}\tau)-({\rm d}_{k-1}\sigma_h,{\rm d}_{k-1}\tau),\\
    (\zeta^u_i,v)+(\nabla \zeta^u_i,\nabla v)&=(f,v)-({\rm d}_{k-1}\sigma_h,v)-({\rm d}_ku_h,{\rm d}_kv)
\end{align*}
for all $\xi\in\mathcal{H}^{k-2}_i, \tau\in\mathcal{H}^{k-1}_i, v\in\mathcal{H}^k_i.$ For $k=n$ we have
\begin{equation*}
    \|\sigma-\sigma_h\|^2_{\mathcal{V}^{n-1}}+\|u-u_h\|^2_{\mathcal{V}^n}\approx\sum_{i=1}^N\left(\|\eta^\sigma_i\|^2_{H^1(\Omega_i)}+\|\zeta^\sigma_i\|^2_{H^1(\Omega_i)}\right)+\|f-{\rm d}_{n-1}\sigma_h\|^2.
\end{equation*}
\end{theorem}

\subsection{Elasticity with weakly imposed symmetry}
In this subsection, we consider the linear elasticity using weakly symmetric tensors in $\mathbb{R}^3$.
Let $\Gamma\neq\emptyset,$ $\Sigma=[\mathcal{V}^d]^3\subseteq[H(\text{div},\Omega)]^3,$ and $U=Q=[L^2(\Omega)]^3$. The space $\Sigma$ consists of matrix-valued functions whose rows are contained in $\mathcal{V}^d.$ Let $\lambda, \mu\in L^\infty(\Omega)$ be the Lam\'e parameters. The elasticity compliance tensor is 
\begin{equation*}
    A\tau=\frac{1}{2\mu}\left(\tau-\frac{\lambda}{2\mu+3\lambda}(\text{Tr}\tau)I\right),\quad\tau\in\Sigma,
\end{equation*}
where $\text{Tr}\tau$ denotes the trace of the matrix $\tau$, and
$I \in \mathbb{R}^{3\times3}$ is the identity matrix. Let $\text{div}$
denote the row-wise divergence operator. Define 
\begin{equation*}
    \text{skw}\tau=(\tau_{23}-\tau_{32},\tau_{31}-\tau_{13},\tau_{12}-\tau_{21})
\end{equation*}
to be the operation of
taking the skew-symmetric part of matrices. The $L^2$ adjoint of skw is defined for $q=(q_1,q_2,q_3)$ as  
\begin{equation*}
\text{skw}^*q=\begin{pmatrix}0&q_3&-q_2\\
-q_3&0&q_1\\
q_2&-q_1&0\end{pmatrix}.    
\end{equation*}
The variational
formulation of linear elasticity with weakly symmetric tensor is: Find
$(\sigma,u,p)\in\Sigma\times U\times Q$ such that
\begin{equation}\label{we}
    \begin{aligned}
    (A\sigma,\tau)+(\text{div}\tau,u)+(\text{skw}\tau,p)&=0,\quad\tau\in\Sigma,\\
    (\text{div}\sigma,v)+(\text{skw}\sigma,q)&=(f,v),\quad (v,q)\in U\times Q.
\end{aligned}
\end{equation}
The stable discretization $(\sigma_h,u_h,p_h)\in\Sigma_h\times U_h\times Q_h\subset\Sigma\times U\times Q$ for \eqref{we} has been well-established in the literature \cite{ABD1984,Stenberg1988}
\begin{equation}\label{diswe}
    \begin{aligned}
    (A\sigma_h,\tau)+(\text{div}\tau,u_h)+(\text{skw}\tau,p_h)&=0,\quad\tau\in\Sigma_h,\\
    (\text{div}\sigma_h,v)+(\text{skw}\sigma_h,q)&=(f,v),\quad (v,q)\in U_h\times Q_h.
\end{aligned}
\end{equation}
For instance, $\Sigma_h\times U_h=[\mathcal{V}_h^d]^3\times [L_h^2]^3$ could be the Brezzi--Douglas--Marini pair and $Q_h=U_h.$ 
The residual of \eqref{diswe} consists of
\begin{equation}\label{RSUQ}
\begin{aligned}
    &\langle\mathcal{R}_\Sigma,\tau\rangle=-(A\sigma_h,\tau)-(\text{div}\tau,u_h)-(\text{skw}\tau,p_h),\quad\forall\tau\in\Sigma,\\
    &\mathcal{R}_U=f-\text{div}\sigma_h,\quad\mathcal{R}_Q=-\text{skw}\sigma_h.
    \end{aligned}
\end{equation}
The error-residual relation of \eqref{diswe} reads
\begin{equation}\label{reserrorE}
\begin{aligned}
    &\|\sigma-\sigma_h\|^2_\Sigma+\|u-u_h\|^2_U+\|p-p_h\|^2_Q\\
    &\approx\langle\mathcal{R}_\Sigma\times\mathcal{R}_U\times\mathcal{R}_Q,\mathcal{B}_{\Sigma\times U\times Q}\mathcal{R}_\Sigma\times\mathcal{R}_U\times\mathcal{R}_Q\rangle,
    \end{aligned}
\end{equation}
where the constants hidden in the equivalence are independent of $\lambda$, see, \cite{LonsingVerfurth2004}.
With the identification \eqref{identification} and equivalence \eqref{Bkk} ($k=n-1$, $n=3$), we have
\begin{equation}\label{BSigma}
\mathcal{B}_{\mathcal{V}^d}\approx\nabla\times\mathcal{B}_{[\mathcal{H}_\Gamma]^3}(\nabla\times)^*+\mathcal{I}_{[\mathcal{H}_\Gamma]^3}\mathcal{B}_{[\mathcal{H}_\Gamma]^3}\mathcal{I}_{[\mathcal{H}_\Gamma]^3}^*,
\end{equation}
where $\nabla\times: [\mathcal{H}_\Gamma]^3\rightarrow\mathcal{V}^d$ and  $\mathcal{I}_{[\mathcal{H}_\Gamma]^3}: [\mathcal{H}_\Gamma]^3\rightarrow\mathcal{V}^d$ is the inclusion.
It follows from \eqref{reserrorE}, \eqref{BSigma} and
\begin{equation}\label{BE}
    \mathcal{B}_{\Sigma\times U\times Q}=\mathcal{B}_\Sigma\times\mathcal{B}_U\times\mathcal{B}_Q=[\mathcal{B}_{\mathcal{V}^d}]^3\times\text{id}_U\times\text{id}_Q,
\end{equation}
that
\begin{equation}\label{weerror}
\begin{aligned}
    &\|\sigma-\sigma_h\|^2_\Sigma+\|u-u_h\|^2_U+\|p-p_h\|^2_Q\\
    &\approx\|(\nabla\times)^*\mathcal{R}_\Sigma\|^2_{\mathcal{B}_{\mathcal{H}_\Gamma}}+\|\mathcal{I}^*_{[\mathcal{H}_\Gamma]^3}\mathcal{R}_\Sigma\|^2_{\mathcal{B}_{\mathcal{H}_\Gamma}}+\|\mathcal{R}_U\|^2+\|\mathcal{R}_Q\|^2.
    \end{aligned}
\end{equation}
Here $\nabla\times: [\mathcal{H}_\Gamma]^{3\times3}\rightarrow\Sigma$ is the row-wise curl in \eqref{weerror}.
Using \eqref{RSUQ} and integration-by-parts, we have
\begin{equation*}
\begin{aligned}
    &\langle (\nabla\times)^*\mathcal{R}_\Sigma,\xi\rangle=(R^E_1,\xi)+(J^E_1,\xi\times\nu)_{\mathcal{S}_h},\quad\forall \xi\in[\mathcal{H}_\Gamma]^3,\\
    &\langle \mathcal{I}^*_{[\mathcal{H}_\Gamma]^3}\mathcal{R}_\Sigma,\tau\rangle=(R^E_2,\tau)+(J^E_2,\tau\cdot\nu)_{\mathcal{S}_h},\quad\forall \tau\in[\mathcal{H}_\Gamma]^3,
    \end{aligned}
\end{equation*}
where for all $T\in\mathcal{T}_h$ and $S\in\mathcal{S}_h$,
\begin{align*}
    &R_1^E|_T=-\nabla\times (A\sigma_h+\text{skw}^*p_h)|_T,\quad J^E_1|_S=-\llbracket A\sigma_h+\text{skw}^*p_h\rrbracket\times\nu|_S,\\
    &R_2^E|_T=(-A\sigma_h+\nabla u_h-\text{skw}^*p_h)|_T,\quad J^E_2|_S=-\llbracket u_h\rrbracket|_S.
\end{align*}
A combination of \eqref{weerror} and Lemma \ref{resH1} yields a residual estimator.
\begin{theorem}
There exist $C_{E,1}>0$, $C_{E,2}>0$ dependent only on $\mu, \Omega, \Gamma$ such that
\begin{align*}
    &\|\sigma-\sigma_h\|_\Sigma+\|u-u_h\|_U+\|p-p_h\|_Q\\
    &\leq C_{E,1}\left\{\sum_{i=1}^2\left(\|hR_i^E\|+\|h_s^\frac{1}{2}J_i^E\|_{\mathcal{S}_h}\right)+\|\mathcal{R}_U\|+\|\mathcal{R}_Q\|\right\},
\end{align*}
and
\begin{align*}
    &C_{E,2}\left\{\sum_{i=1}^2\left(\|hR_i^E\|+\|h_s^\frac{1}{2}J_i^E\|_{\mathcal{S}_h}\right)+\|\mathcal{R}_U\|+\|\mathcal{R}_Q\|\right\}\\
    &\leq\|\sigma-\sigma_h\|_\Sigma+\|u-u_h\|_U+\|p-p_h\|_Q+\emph{osc}_h((\nabla\times)^*\mathcal{R}_\Sigma)+\emph{osc}_h(\mathcal{I}_{[\mathcal{H}_\Gamma]^3}^*\mathcal{R}_\Sigma).
\end{align*}
\end{theorem}
The same estimator could be found in \cite{LonsingVerfurth2004} under the assumption that the domain is convex.  The work \cite{Kim2011} derives an equilibrated estimator for \eqref{diswe} with guaranteed upper bound in two dimension.

Recall that $\mathcal{H}_h^1=\mathcal{H}_h^{0,1}\subset\mathcal{H}_\Gamma$ is the subspace of continuous and scalar-valued piecewise linear polynomials. Finally we present an implicit error estimator using \eqref{weerror} and a vector-valued version of Lemma \ref{esteta}. 
\begin{theorem}
Assume $\Sigma_h\supseteq[\mathcal{H}_h^1]^{3\times3}$. Then we have
\begin{align*}
    &\|\sigma-\sigma_h\|^2_\Sigma+\|u-u_h\|^2_U+\|p-p_h\|^2_Q\\
    &\approx\sum_{i=1}^N\left(\|\eta_i\|^2_{H^1(\Omega_i)}+\|\zeta_i\|^2_{H^1(\Omega_i)}\right)+\|f-\emph{div}\sigma_h\|^2+\|\emph{skw}\sigma_h\|^2,
\end{align*}
where $\eta_i\in[\mathcal{H}_i]^{3\times3}$, $\zeta_i\in[\mathcal{H}_i]^{3\times3}$ solve
\begin{align*}
    (\eta_i,\xi)+(\nabla \eta_i,\nabla\xi)&=-(A\sigma_h,\nabla\times\xi)-(\emph{skw}\nabla\times\xi,p_h),\quad\forall\xi\in[\mathcal{H}_i]^{3\times3},\\
    (\zeta_i,\tau)+(\nabla \zeta_i,\nabla\tau)&=-(A\sigma_h,\tau)-(\emph{div}\tau,u_h)-(\emph{skw}\tau,p_h),\quad\forall\tau\in[\mathcal{H}_i]^{3\times3}.
\end{align*}
\end{theorem}

\begin{remark}
{In this section, a posteriori error estimates are derived from localized continuous Riesz representation in product form, while block diagonal preconditioners for solving saddle point systems are based on discrete product Riesz representation (cf.~\cite{MardalWinther2011}). In many practical applications, it is known that block  triangular preconditioners are superior to their diagonal counterparts. The mathematical theory of triangular preconditioners is based on the Field-of-Value (FOV) analysis (cf.~\cite{Elman1982,LoghinWathen2004,Wathen2015,Ma2016,LiZikatanovZuo2024arXiv}), which  is applicable to non-symmetric system of PDEs. Therefore, another future research direction is  relating a posteriori error analysis of linear PDEs, e.g., convection-diffusion and Helmholtz equations, with FOV preconditioning on the continuous level.}
\end{remark}

\section*{Acknowledgments}
The authors would like to thank Han Shui for the help with the numerical experiment and the anonymous referees for helpful remarks that improved the quality of this paper. The work of Li was supported by the National Natural Science Foundation of China under grant 12471346 and the Fundamental Research Funds for the Central Universities under grant 226-2023-00039. The work of Zikatanov was supported in part by the U.~S.~National Science Foundation (DMS-2208249) and the U.~S.-Norway Fulbright Foundation.

\appendix
\section*{Appendix}\setcounter{section}{1}\setcounter{equation}{0}
{\normalsize 
We construct $\Pi_h^k$ in Lemma \ref{Pik} following the same idea in \cite{Clement,ScottZhang1990}. However, we regularize $v\in H^1\Lambda^k(\Omega)$ on $n$-dimensional elements except for degrees of freedom related to $\Gamma.$ As a consequence, the interpolation is $L^2$ bounded in $\mathcal{H}^k$, which is essential for deriving error estimators for singularly perturbed problems.
\subsection*{Proof of Lemma \ref{Pik}}
Let $\{\Delta_i\}_{i=1}^M$ be the set of $k$-dimensional simplexes in $\mathcal{T}_h$. Given a sufficiently smooth $k$-form $v$, the degrees of freedom of $\mathcal{P}_1^-\Lambda^k(\mathcal{T}_h)$ consists of $\int_{\Delta_i}\text{tr}v$ with $1\leq i\leq M.$ Let $\{\phi_i\}_{i=1}^M$ be the corresponding dual basis of $\mathcal{V}_h^{k,0}.$ Each $\Delta_i$ is assigned with an $n$- or $(n-1)$-dimensional simplex $\sigma_i\supseteq\Delta_i$ in $\mathcal{T}_h.$ In particular, let $\sigma_i$ be an element in $\mathcal{T}_h$ for $\Delta_i\not\subset\Gamma$ and be a face in $\Gamma$ for $\Delta_i\subset\Gamma$. Let $Q_{\sigma_i}$ denote the $L^2$ projection onto $\mathcal{P}_0\Lambda^k(\sigma_i).$ The interpolation operator $\Pi_h^k: H^1\Lambda^k(\Omega)\rightarrow\mathcal{P}^-_1\Lambda^k(\mathcal{T}_h)$ is defined as
\begin{equation*}
    \Pi_h^kv=\sum_{i=1}^M\left(\int_{\Delta_i}\text{tr}Q_{{\sigma_i}}v\right)\phi_i.
\end{equation*}
By construction, we have for $T\in\mathcal{T}_h$,
\begin{enumerate}
    \item $v\in\mathcal{P}_0\Lambda^k(\Omega_T^k)\Longrightarrow$ $\Pi_h^kv=v$ on $T$,
    \item  $\text{tr}|_\Gamma v=0\Longrightarrow$ $\text{tr}|_\Gamma\Pi_h^kv=0$ and $\|\Pi_h^kv\|_T\preccurlyeq\|v\|_{\Omega_T^k}$.
\end{enumerate}
The approximation in Lemma \ref{approxPik} follows from the Bramble--Hilbert lemma and that $\Pi_h^k$ preserves constants locally.\qed

\subsection*{Proof of Theorem \ref{robustHodge}}
Using the Hodge decomposition \eqref{Hodge} and the Poincar\'e inequality \eqref{Poincare}, we have 
\begin{equation}\label{Hodgev}
    v={\rm d}_{k-1}\varphi_0+z_0, 
\end{equation}
where $\varphi_0\in\mathfrak{Z}^{k-1\perp}\subset\mathcal{V}^{k-1}$,  $z_0\in\mathfrak{H}^k\oplus\mathfrak{Z}^{k\perp}$, and 
\begin{equation}\label{bdphiz0}
\begin{aligned}
&\|\varphi_0\|\preccurlyeq\|{\rm d}_{k-1}\varphi_0\|\leq\|v\|,\\
    &\|z_0\|\leq\|v\|,\quad\|{\rm d}_kz_0\|=\|{\rm d}_kv\|.
\end{aligned}
\end{equation}

We shall extend $z_0$ to a slightly larger domain with smooth boundary.  
In fact, there exist a Lipschitz domain $\Omega^e\supset\Omega$, an extension operator $E^k: \mathcal{V}^k\rightarrow H\Lambda^k(\Omega^e)$, and an outer Lipschitz neighborhood $\Omega_\Gamma\subset\Omega^e$ of $\Gamma$ (see, e.g., \cite{GopQiu2012,Licht2019,Schoberl2008}), such that for all $v\in\mathcal{V}^k$,
\begin{equation}\label{Ek}
    \begin{aligned}
    & E^kv|_\Omega=v,\quad
    E^kv|_{\Omega_\Gamma}=0,\\
    & \|E^kv\|_{\Omega^e}\lesssim\|v\|,\quad    {\rm d}_kE^kv =E^{k+1}{\rm d}_{k}v.
\end{aligned}
\end{equation}
Without loss of generality, we assume the boundary of $\Omega^e$ is smooth by restricting $E^kv$ to a slightly smaller but smooth subdomain of $\Omega^e$.
We refer the reader to Figure~\ref{fig:extend} for sketch of the domains involved in the construction.

Now, let $\tilde{z}=E^kz_0$. It follows from \eqref{Ek} that
\begin{subequations}
\begin{align}
  &\|\tilde{z}\|_{\Omega^e}\preccurlyeq\|z_0\|,\quad\|{\rm d}\tilde{z}\|_{\Omega^e}\preccurlyeq\|{\rm d}_kz_0\|,\label{bdztilde}\\
  &\tilde{z}=0\quad\text{ on } \Omega_\Gamma.\label{ztilde0}
\end{align}
\end{subequations}
Using the Hodge decomposition on $\Omega^e$ gives that
\begin{equation}\label{Hodgeztilde}
    \tilde{z}={\rm d}_{k-1}\tilde{\varphi}+\tilde{w}\quad\text{ on }\Omega^e,
\end{equation}
where $\tilde{\varphi}\in N({\rm d}_{k-1}|_{\Omega^e})^\perp$ and $\tilde{w}\in R({\rm d}_{k-1}|_{\Omega^e})^\perp$.
Since $\partial\Omega^e$ is smooth and \eqref{qzH1}, we have
$\tilde{\varphi}\in H^1\Lambda^{k-1}(\Omega^e)$,
$\tilde{w}\in H^1\Lambda^k(\Omega^e)$, and
\begin{equation}\label{bdphiwtilde}
\begin{aligned}
    \|\tilde{\varphi}\|_{H^1(\Omega^e)}&\preccurlyeq\|\tilde{z}\|_{\Omega^e},\\
    \|\tilde{w}\|_{\Omega^e}&\preccurlyeq\|\tilde{z}\|_{\Omega^e},\\
    |\tilde{w}|_{H^1(\Omega^e)}&\preccurlyeq\|\tilde{z}\|_{\Omega^e}+\|{\rm d}_k\tilde{z}\|_{\Omega^e}.
    \end{aligned}
\end{equation}

\begin{figure}[!htb]
  \centering
\includegraphics[width=0.4\textwidth]{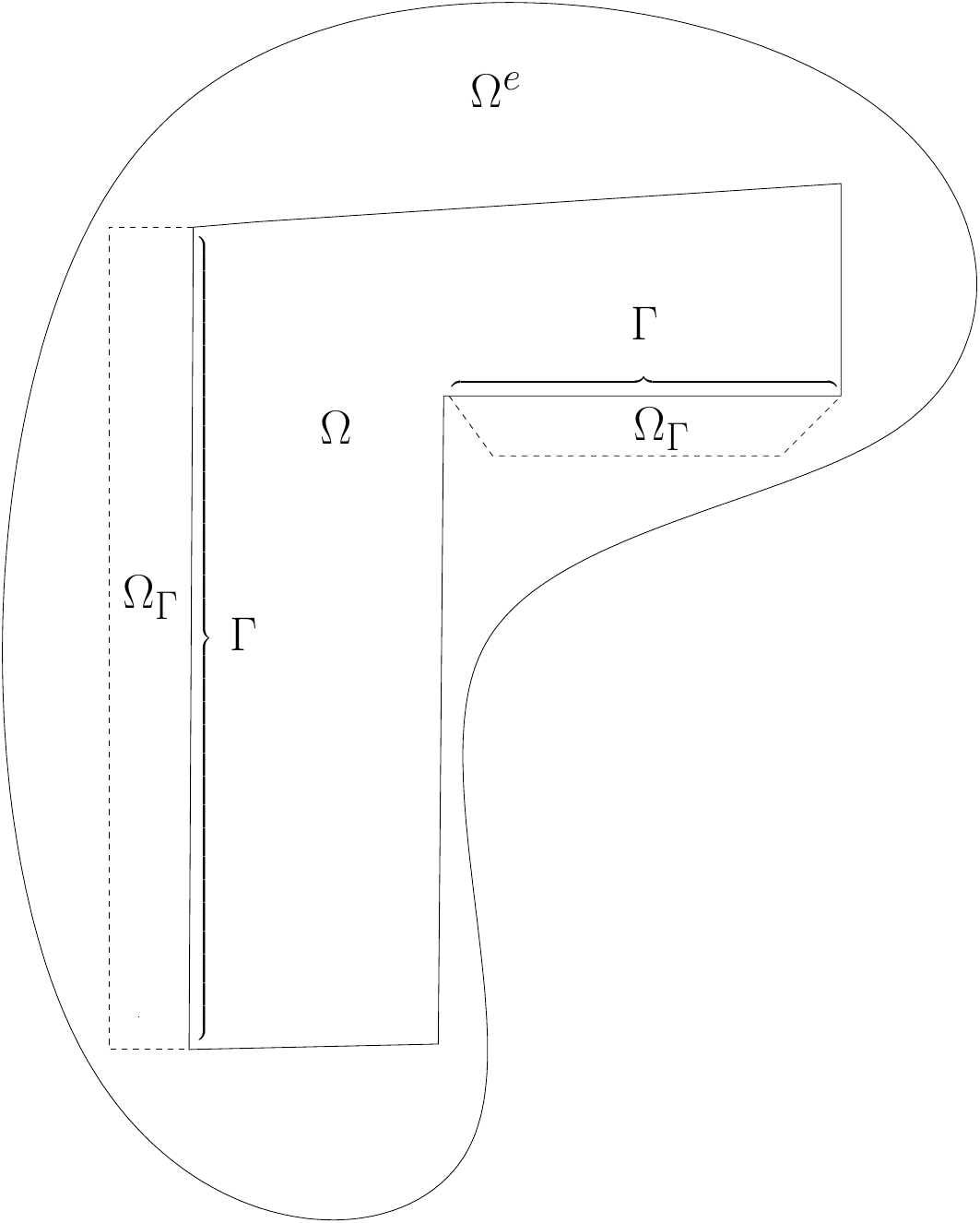}
  \caption{An extended domain $\Omega^e\supset\Omega$ and an outer neighborhood $\Omega_\Gamma$ of $\Gamma$.\label{fig:extend}}
\end{figure}

If $\Gamma=\emptyset$, then we complete the proof by taking $\varphi=\varphi_0+\tilde{\varphi}|_\Omega\in\mathcal{V}^{k-1}$ and $z=\tilde{z}|_\Omega\in\mathcal{H}^k$. In general, $\tilde{\varphi}$, $\tilde{w}$ must be modified to satisfy the homogeneous boundary condition on $\Gamma.$ Using \eqref{ztilde0}, it follows that ${\rm d}_{k-1}\tilde{\varphi}=-\tilde{w}\in H^1\Lambda^k(\Omega_\Gamma)$ in the subdomain $\Omega_\Gamma$. Then by the universal extension theorem in \cite{HipLiZou2012}, we can extend {$\tilde{\varphi}|_{\Omega_\Gamma}\in H^1\Lambda^{k-1}(\Omega_\Gamma)$} with ${\rm d}_{k-1}\tilde{\varphi}|_{\Omega_\Gamma}\in H^1\Lambda^k(\Omega_\Gamma)$ to obtain $\bar{\varphi}\in H^1\Lambda^{k-1}(\mathbb{R}^n)$ with {${\rm d}_{k-1}\bar{\varphi}\in H^1\Lambda^k(\mathbb{R}^n)$}, and
\begin{equation}\label{extphibar}
\begin{aligned}
    \|\bar{\varphi}\|+\|{\rm d}_{k-1}\bar{\varphi}\|&\preccurlyeq\|\tilde{\varphi}\|_{\Omega_\Gamma}+\|\tilde{w}\|_{\Omega_\Gamma},\\
    \|\bar{\varphi}\|_{H^1}+\|{\rm d}_{k-1}\bar{\varphi}\|_{H^1}&\preccurlyeq\|\tilde{\varphi}\|_{H^1(\Omega_\Gamma)}+\|\tilde{w}\|_{H^1(\Omega_\Gamma)}.
    \end{aligned}
\end{equation}
Using \eqref{Hodgev} and \eqref{Hodgeztilde}, we rewrite  $v$ for  $x\in \Omega$ as
\begin{equation}\label{almost}
    v=
    {\rm d}_{k-1}\varphi_1+z,
\end{equation}
where $\varphi_1=\varphi_0+\tilde{\varphi}|_\Omega-\bar{\varphi}|_\Omega\in\mathcal{V}^{k-1}$, $z={\rm d}_{k-1}\bar{\varphi}|_\Omega+\tilde{w}|_\Omega\in\mathcal{H}^k$. Here we have $z=0$ on $\Gamma$ because ${\rm d}_{k-1}\tilde{\varphi}+\tilde{w}=0$ on $\Omega_\Gamma$, $\Gamma\subset\partial\Omega_\Gamma$, and ${\rm d}_{k-1}\bar{\varphi}={\rm d}_{k-1}\tilde{\varphi}$ on $\Gamma$.
Collecting previous bounds \eqref{bdphiz0},  \eqref{bdztilde},  \eqref{bdphiwtilde}, \eqref{extphibar} then shows that
\begin{equation*}
\begin{aligned}
    &\|\varphi_1\|_{\mathcal{V}^{k-1}}+\|z\|\preccurlyeq\|v\|,\quad\\
    &|z|_{H^1}\preccurlyeq\|v\|+\|{\rm d}_kv\|. 
    \end{aligned}
\end{equation*}
Finally, the regular decomposition of $\varphi_1$ implies that there
exists $\varphi\in\mathcal{H}^{k-1}$ satisfying
\begin{equation*}
    {\rm d}_{k-1}\varphi={\rm d}_{k-1}\varphi_1,\quad\|\varphi\|_{H^1}\preccurlyeq\|{\rm d}_{k-1}\varphi_1\|\preccurlyeq\|v\|.
\end{equation*}
Replacing $\varphi_1$ with $\varphi$ in \eqref{almost} completes the proof.
\qed

\bigskip
If the space of harmonic forms $\mathfrak{H}^k=\{0\}$ is trivial, the component $z$ in Theorem \ref{robustHodge} can be chosen such that
\begin{equation*}
    \|z\|\preccurlyeq\|v\|,\quad|z|_{H^1}\preccurlyeq\|{\rm d}_kv\|. 
\end{equation*}
In the previous proof for
$z_0\in\mathfrak{H}^k\oplus\mathfrak{Z}^{k\perp}$, let us consider the
decomposition $z_0=q+w_0$ where $q\in\mathfrak{H}^k$,
$w_0\in\mathfrak{Z}^{k\perp}$. If, in the arguments in this proof, we replace $z_0$ with $w_0$, then we arrive at the following
corollary, which might be useful for the analysis of other methods and
techniques.
\begin{corollary}
For any $v\in\mathcal{V}^k$, there exist $z\in \mathcal{H}^k$ such that 
\begin{equation*}
    {\rm d}_kv={\rm d}_kz,\quad
    \|z\|\preccurlyeq\|v\|,\quad\mbox{and}\quad    |z|_{H^1}\preccurlyeq \|{\rm d}_kv\|.
\end{equation*}
\end{corollary}

\begin{remark}
Unlike the pure Dirichlet ($\Gamma=\partial\Omega$) or Neumann ($\Gamma=\emptyset$) boundary condition, the structure of $\mathfrak{H}^k$ under the mixed boundary condition is dependent on the relative homology of the pair $(\Omega,\Gamma)$. It turns out that $\mathfrak{H}^k$ could be nontrivial even if $\Omega$ is star-shaped; see, e.g., \cite{GolMM2011,Licht2019}.
\end{remark}
}

\bibliographystyle{amsplain}   

\providecommand{\bysame}{\leavevmode\hbox to3em{\hrulefill}\thinspace}
\providecommand{\MR}{\relax\ifhmode\unskip\space\fi MR }
\providecommand{\MRhref}[2]{%
  \href{http://www.ams.org/mathscinet-getitem?mr=#1}{#2}
}
\providecommand{\href}[2]{#2}

\end{document}